\documentclass[reqno,dvipsnames,12pt]{amsart}
\usepackage{geometry}
\usepackage{subcaption}
\usepackage{amssymb, amsmath}
\usepackage{mathrsfs}
\usepackage{amscd} 
\usepackage{amsmath}
\usepackage[normalem]{ulem}
\usepackage[english]{babel}

\usepackage{lmodern}      
\usepackage{float} 
\usepackage{stmaryrd}
 \usepackage{hyperref}
\hypersetup{
	colorlinks   = true,          
	urlcolor     = blue,          
	linkcolor    = blue,       
	citecolor   = red            
}
\usepackage{enumitem}
\setenumerate[1]{label=(\roman*)} 
\usepackage[normalem]{ulem}
\usepackage{graphicx}
\usepackage[curve]{xypic}
\usepackage{tikz}
\usetikzlibrary{arrows,decorations,shapes,shadows}
\usetikzlibrary{calc,intersections}
\usepackage{verbatim}

\newbox\mybox
\def\overtag#1#2#3{\setbox\mybox\hbox{$#1$}\hbox to
  0pt{\vbox to 0pt{\vglue-#3\vglue-\ht\mybox\hbox to \wd\mybox
      {\hss$\ss#2$\hss}\vss}\hss}\box\mybox}
\def\undertag#1#2#3{\setbox\mybox\hbox{$#1$}\hbox to 0pt{\vbox to
    0pt{\vglue#3\vglue\ht\mybox\hbox to \wd\mybox
      {\hss$\ss#2$\hss}\vss}\hss}\box\mybox}
\def\lefttag#1#2#3{\hbox to 0pt{\vbox to 0pt{\vglue -6pt\hbox to
      0pt{\hss$\ss#2$\hskip#3}\vss}}#1}
\def\righttag#1#2#3{\hbox to 0pt{\vbox to 0pt{\vglue -6pt\hbox to
      0pt{\hskip#3$\ss#2$\hss}\vss}}#1}
\let\ss\scriptstyle
\def\splicediag#1#2{\xymatrix@R=#1pt@C=#2pt@M=0pt@W=0pt@H=0pt}
\def\Dot{\lower.2pc\hbox to 2pt{\hss$\bullet$\hss}}
\def\Circ{\lower.2pc\hbox to 2pt{\hss$\circ$\hss}}
\def\Vdots{\raise5pt\hbox{$\vdots$}}
\newcommand\lineto{\ar@{-}}
\newcommand\dashto{\ar@{--}}
\newcommand\dotto{\ar@{.}}

\renewcommand{\setminus}{\smallsetminus}

\newtheorem*{theorem*}{Theorem}
\newtheorem*{corollary*}{Corollary}

\newtheorem{theorem}{Theorem}[section]

\newtheorem{Corollary}[theorem]{Corollary}

\newtheorem{thm}{Theorem}[section]

\newtheorem{prop}[thm]{Proposition}
\newtheorem{lemm}[thm]{Lemma}
\newtheorem{coro}[thm]{Corollary}
\newtheorem{rema}[thm]{Remark}
\newtheorem{nota}[thm]{Notation}

\newtheorem{defi}[thm]{Definition}

\newtheorem{exam}[thm]{Example}

\usepackage{xcolor}

\makeatletter
\@namedef{subjclassname@2020}{\textup{2020} Mathematics Subject Classification}
\makeatother

\title{Lipschitz geometry of complex surface germs via inner rates of primary ideals}
\author{Yenni Cherik}
\address{Aix-Marseille Universit\'e, CNRS, Centrale Marseille, I2M, Marseille, France}
\date{}
\email{\href{mailto:yenni.cherik@univ-amu.fr}{yenni.cherik@univ-amu.fr}}

\begin{document}



\maketitle

\begin{abstract}
Let $(X, 0)$ be a normal complex surface germ embedded in $(\mathbb{C}^n, 0)$, and denote by $\mathfrak{m}$ the maximal ideal of the local ring $\mathcal{O}_{X,0}$. In this paper, we associate to each $\mathfrak{m}$-primary ideal $I$ of $\mathcal{O}_{X,0}$ a continuous function $\mathcal{I}_I$ defined on the set of positive (suitably normalized) semivaluations of $\mathcal{O}_{X,0}$. We prove that the function $\mathcal{I}_{\mathfrak{m}}$ is determined by the outer Lipschitz geometry of the surface $(X, 0)$. We further demonstrate that for each $\mathfrak{m}$-primary ideal $I$, there exists a complex surface germ $(X_I, 0)$ with an isolated singularity whose normalization is isomorphic to $(X, 0)$ and $\mathcal{I}_I = \mathcal{I}_{\mathfrak{m}_I}$, where $\mathfrak{m}_I$ is the maximal ideal of $\mathcal{O}_{X_I,0}$. Subsequently, we construct an infinite family of complex surface germs with isolated singularities, whose normalizations are isomorphic to $(X,0)$ (in particular, they are homeomorphic to $(X,0)$) but have distinct outer Lipschitz types.
 \end{abstract}
 
\tableofcontents

\section*{Quick overview}

Let $(X,0)$ be a germ of complex space embedded in $(\mathbb{C}^n,0)$. It is well known that the \textit{topological type} of $(X,0)$, that is its class of homeomorphism, determines and is determined by the homeomorphism class of its link $X^{(\epsilon)} = X \cap S_{\epsilon}^{2n-1}$ for $\epsilon >0$ sufficiently small, where $S_{\epsilon}^{2n-1}$ is the sphere of radius $\epsilon$ . Thanks to this result, the topology of germs of complex spaces is quite well understood; however, it is not the case for their metric type, and that is what we are going to investigate in this paper.

 Let us first give some historical context: the standard metric of $\mathbb{C}^n$ induces two natural distances on $(X,0)$. The \textit{outer metric}, denoted $d_o$, is defined by restricting the Euclidean distance of $\mathbb{C}^n$ to $X$, and the \textit{inner metric}, referred to as $d_i$, quantifies the lengths of arcs on $X$. There are many motivations to study these metrics. Among them is the fact that the biholomorphic class of the germ $(X,0)$ determines the inner (resp. outer) \textit{Lipschitz geometry} of $(X,0)$ (see Definition \ref{bilequiv}). Another motivation arises from a result by Mostowski in \cite{Mostowski1985a}, which asserts that the set of complex germs up to bilipschitz equivalence with respect to the inner (resp. outer) metric is countable. This result was significantly generalized by Parusi\'nski in \cite{parulipschitz2} for the real case and then Nhan and Valette in \cite{valette1} for more general O-minimal structures.

  The complete classification of complex germs under bilipschitz equivalence is then more attainable compared to the analytic classification and, as expected, richer than the topological classification. However, in the case of germs of complex curves, the difference between Lipschitz geometry and topology does not appear. Indeed, a germ of curve $(C,0)$ embedded in $\mathbb{C}^n$ is metrically conical, meaning that it is inner bilipschitz equivalent to the cone over its link $C^{\epsilon} = C \cap \mathbb{S}_{\epsilon}$. As for the outer metric, it follows from the work of Teissier in \cite{Teissier1982} that any curve germ $(C,0) \subset (\mathbb{C}^n,0)$ is outer bilipschitz equivalent to a germ of a complex plane curve (embedded in $\mathbb{C}^2$). Later, Neumann and Pichon proved in \cite{NeumannPichon2014} that the outer Lipschitz geometry of a plane curve germ $(C,0) \subset (\mathbb{C}^2,0)$ determines and is determined by the embedded topological type of the curve germ.

   More recently, Neumann, Birbrair, and Pichon gave a complete classification of complex surface germs with isolated singularities up to bilipschitz equivalence with respect to the inner metric by means of the \textit{inner rates of complex surface germs} (see Definition \ref{genericBFP}), a topic we elaborate on in this paper. Furthermore, Neumann and Pichon's work in \cite{NP2016} establishes a connection between the outer metric of a complex surface germ $(X,0) \subset (\mathbb{C}^n,0)$ and the families of \textit{generic hyperplane sections} and  \textit{generic polar curves} of $(X,0)$ (see Definitions \ref{polar curve} and \ref{genericBFP}). Their result roughly states that the outer Lipschitz geometry of $(X,0)$ determines the relative position between these families of curves.

    In an early study, Belotto, Fantini, Némethi, and Pichon proved in \cite{BFNP} that when the homeomorphism class of a normal surface germ $(X,0)$ is fixed, there are, up to homeomorphism, only finitely many possible relative positions between the family of generic hyperplane sections and the family of polar curves associated with the generic projections on $(X,0)$.

     In this paper, we investigate what happens to the last mentioned result when we do not assume the normality of the complex surface germ $(X,0)$. Our result (Theorem \ref{thmM}) essentially states that when the homeomorphism class of the germ $(X,0)$ and the analytic type of its \textit{normalization} $(\overline{X},0)$ are fixed, there are, up to homeomorphism, infinitely many possible relative positions between the generic hyperplane sections and the generic polar curves of $(X,0)$. In particular, it follows from our result and the work of Neumann and Pichon described earlier that there are infinitely many outer Lipschitz geometries on complex surface germs even when we fix the homeomorphism class of the germs and the analytic type of their normalization.

\section*{Content of the paper}

In order to state our result precisely, we need to introduce some terminology first. Let $\pi: (X_{\pi}, E) \longrightarrow (X,0)$ be a \textit{good resolution} of $(X,0)$, that is, a proper bimeromorphic map which is an isomorphism outside of a simple normal crossing divisor $\pi^{-1}(0) = E$, called the \textit{exceptional divisor}. Let $\Gamma_{\pi}$ be the \textit{dual graph} associated with $\pi$, that is, the graph whose vertices are in bijection with the irreducible components of $E$ and such that the edges between two vertices $v$ and $v'$ correspond to the intersecting points between the components $E_v$ and $E_{v'}$. Each vertex $v$ of this graph is weighted with the self-intersection number $E_v^2$ and the genus $g_v$ of the corresponding compact complex curve $E_v$. Given a holomorphic function $f: (X,0) \longrightarrow (\mathbb{C},0)$, we denote by $f^*$ the \textit{strict transform} of $f$ by $\pi$, that is, the curve $f^* := \overline{\pi^{-1}(f^{-1}(0)) \backslash E}$. The following notations were introduced in \cite{BFNP}: given a vertex $v$ of $\Gamma_{\pi}$, we denote by $\ell_v$ the intersection multiplicity $\ell^* \cdot E_v$, where $\ell: (X,0) \longrightarrow (\mathbb{C},0)$ is a \textit{generic linear form} in the sense of \cite{BFP} (see Definition \ref{genericBFP}). We denote by $p_v$ the intersection multiplicity $\Pi \cdot E_v$ where $\Pi$ is the polar curve associated with a \textit{generic linear projection} $P: (X,0) \longrightarrow (\mathbb{C}^2,0)$ again in the sense of \cite{BFP} (see Definition \ref{genericBFP}). The \textit{blow up of the maximal ideal} (see Definition \ref{blowupoideal}) is the minimal modification which resolves the family of \textit{generic hyperplane sections} (See Definition \ref{genericBFP}) of $(X,0)$. The \textit{Nash transform} (see Definition \ref{classicnashmodif}) is the minimal modification which resolves the family of \textit{generic polar curves}, that is, the polar curves associated with the generic plane linear projections. We associate to the minimal good resolution $\pi: (X_{\pi}, E) \longrightarrow (X,0)$, which factors through the blow up of the maximal ideal  and the Nash transform  a triple $(\Gamma_{\pi}, L_{\pi}, P_{\pi})$ where $L_{\pi} = (\ell_v)_{v \in V(\Gamma_{\pi})}$ and $P_{\pi} = (p_v)_{v \in V(\Gamma_{\pi})}$, where $V(\Gamma_{\pi})$ denotes the set of vertices of $\Gamma_{\pi}$. Before stating our result, we will first recall a result of Belotto, Fantini, N\'emethi, and Pichon, which will give a bit more context.

 \begin{theorem*}[{\cite[Theorem A]{BFNP}}]\label{explorationpolairenemethi}
	Let $M$ be a real $3$-manifold. There exists finitely many triple $(\Gamma,L,P)$, where $\Gamma$ is a weighted graph and $L$ and $P$ are vectors in $(\mathbb{Z}_{\geq 0})^{V(\Gamma)}$, such that there exists a normal complex surface germ $(X,0)$  satisfying the following conditions:
	\begin{enumerate}
		\item The link of $(X,0)$ is homeomorphic to M.
		\item $(\Gamma,L,P)=(\Gamma_{\pi},L_{\pi},P_{\pi})$.
	\end{enumerate}
		
	\end{theorem*}
In this paper, we investigate the following question: Does the statement of Theorem \cite[Theorem A]{BFNP} still hold if we relax the hypothesis from a normal complex surface germ to a complex surface germ with an isolated singularity ? The following result give a negative answer to this question

 \begin{theorem}[{Theorem \ref{belottofantininemethipichonbis}}]\label{thmM}\label{belotonamathifantinipichon}		Let $(X,0)$ be a normal complex surface germ. Let $\Lambda_X$ be the infinite set of  integrally closed $\mathfrak{m}$-primary ideals of $\mathcal{O}_{X,0}$ .	There exists a family of complex surface germs with an isolated singularity $\{(X_I,0)\}_{I \in \Lambda_X}$ such that

	\begin{enumerate}
		\item $(X,0)$ is homeomorphic to $(X_I,0)$ for all $I \in \Lambda_X$ and is the common normalization of the germs $\{(X_I, 0)\}_{I \in \Lambda_X}$.	
		
		\item For all $I,J$ in $\Lambda_X$ we have that $ (\Gamma_{\pi_I},L_{\pi_I},P_{\pi_I}) = (\Gamma_{\pi_J},L_{\pi_J},P_{\pi_J})$ if and only if  $I=J$. where $\pi_I$ and $\pi_J$ are the minimal good resolutions which factors through the blow-ups of the maximal ideal and the Nash transform of $(X_I,0)$ and $(X_J,0)$ respectively.
		
		 \end{enumerate}

\end{theorem}

We can apply this result to the study of the Lipschitz geometry of complex surface germs by combining it with the following result of Neumann and Pichon:

\begin{theorem*}[{\cite[Theorem 1.2]{NP2016}}]\label{NP16}
		Let $(X,0)$ be a complex surface germ with an isolated singularity. Then the outer Lipschitz geometry of $X$ determines the triple $(\Gamma_{\pi},L_{\pi},P_{\pi})$ where $\pi$ is the minimal good resolution which factors through the blow-up of the maximal ideal of $\mathcal{O}_{X,0}$ and the Nash transform of $(X,0)$.
	\end{theorem*}		
As a consequence of this last result and Theorem \ref{thmM} we get the following

\begin{Corollary}[{Corollary \ref{belottofantinipichonlipschitzbis}}]\label{coroM}

Let $(X,0)$ be a normal complex surface germ. Let $\Lambda_X$ be the  infinite set of integrally closed $\mathfrak{m}$-primary ideals of $\mathcal{O}_{X,0}$.	There exists a family of complex surface germs with an isolated singularity $\{(X_I,0)\}_{I \in \Lambda_X}$ such that

	\begin{enumerate}
		\item $(X,0)$ is homeomorphic to $(X_I,0)$ for all $I \in \Lambda_X$ and is the common normalization of the germs $\{(X_I, 0)\}_{I \in \Lambda_X}$.	
		
		\item For all $I,J \in \Lambda_X$ the surfaces $(X_I,0)$ and $(X_J,0)$ are outer Lipschitz equivalent if and only if $I=J$.	
	\end{enumerate}
\end{Corollary}

Even though it does not appear in the statements of Theorem \ref{thmM} and Corollary \ref{coroM}, the family of surfaces $\{(X_I,0)\}_{I \in \Lambda_X}$ can be explicitly constructed. In what follows, we will introduce the required tools and results needed for such a construction. Besides proving Theorem \ref{thmM}, the following results and constructions have their own independent interest.

We introduce an infinite family of numerical analytic invariants of $(X,0) \subset (\mathbb{C}^n,0)$ associated with an $\mathfrak{m}$-primary ideal $I$ of $\mathcal{O}_{X,0}$: let $\pi: (X_{\pi}, E) \longrightarrow (X,0)$ be a good resolution and $E_v$ be an irreducible component of $E$. The \textit{inner rate} of $I$ along $E_v$ is the rational number
\[ q_v^{I} := \frac{\nu_{v}(I) - m_v(I) + 1}{m_v(I)}, \]
where $m_v(I) := \inf \{ \mathrm{Ord}_{E_v}(h) \mid h \in I \}$, and $\nu_v(I) := \inf \{ \mathrm{Ord}_{E_v}(\pi^{*}(\mathrm{d}h_1 \wedge \mathrm{d}h_2)) \mid h_1, h_2 \in I \}$. These numbers generalize the notion of \textit{inner rates of complex surface germs} (See Definition \ref{genericBFP}) first introduced by Birbrair, Neumann, and Pichon in \cite{BNP} to study the inner Lipschitz geometry of complex surface germs. Indeed, let $\mathcal{N}_v$ be a small tubular neighborhood of $E_v$ with a small neighborhood of its double points removed. Roughly speaking, the inner rate of $(X,0)$ along the component $E_v$ is a rational number $q_v$ that measures how fast the diameter (with respect to the inner metric) of the set $\pi(\mathcal{N}_v) \cap \mathbb{S}_{\epsilon}^{2n-1}$  tends to zero as $\epsilon$ tends to zero. We explain in Remark \ref{maximalidealrate} that the rational number $q_v$, as equivalently defined in \cite{BFNP}, is equal to the number $q_{v}^{\mathfrak{m}}$, where $\mathfrak{m}$ is the maximal ideal of $\mathcal{O}_{X,0}$.\\

Studying the inner rates appears to be much more convenient through valuation theory than solely relying on dual graphs. Let us explain why. The \textit{Non-archimedean link} (See Definition \ref{linknonarchimedien}) is the set of positive and \textit{normalized semivaluations} defined on the completion of the local ring $\mathcal{O}_{X,0}$ endowed with the Tychonoff topology. Among these semivaluations, we have the \textit{divisorial valuations} defined as follows: for each good resolution $\pi$ of $(X,0)$ and for every vertex $v$ of $\Gamma_{\pi}$, we define the valuation $\mathrm{val}_{E_v}(f) := \frac{\mathrm{Ord}_{E_v}(f)}{\mathrm{inf} \{\mathrm{Ord}_{E_v}(h) \mid h \in \mathfrak{m} \}}$, known as the divisorial valuation associated with $E_v$. It turns out that the set of divisorial valuations is dense in $\mathrm{NL}({X,0})$ (See Lemma \ref{densiteremarque}). Each dual graph $\Gamma_{\pi}$ embeds into $\mathrm{NL}(X,0)$ by associating each vertex to its corresponding divisorial valuation. Conversely, $\mathrm{NL}(X,0)$ can be continuously retracted to $\Gamma_{\pi}$. Thus, we can establish a homeomorphism between the non-archimedean link and the inverse limit of the dual graphs obtained from all good resolutions. The topological space $\mathrm{NL}(X,0)$ can thus be viewed as a universal dual graph. This construction was first introduced by Favre and Jonsson in \cite{FavreJonsson2004} for the smooth case. It was generalized to the singular case by Favre in \cite{Favre2010}.

Next, for each $\mathfrak{m}$-primary ideal $I$ of $\mathcal{O}_{X,0}$, we define a real continuous function $\mathcal{I}_I$ on $\mathrm{NL}(X,0)$ such that its value on each divisorial valuation $\mathrm{val}_{E_v}$ of $\mathrm{NL}(X,0)$ is the inner rate $q_v^I$. Furthermore, by equipping $\mathrm{NL}(X,0)$ with an appropriate metric denoted $\mathrm{d}_I$, which depends on the ideal $I$, we demonstrate that the function $\mathcal{I}_I$ is linear on each edge of every graph $\Gamma_{\pi}$ where $\pi$ is a good resolution that factors through the blow up of $I$ (see Definition \ref{blowupoideal}) and the principalization of $\Omega_I^2$ (see Definition \ref{principalization}), generalizing the inner rate function introduced in \cite{BFP}. \\

The following theorem provides a geometric interpretation of the inner rates associated with an $\mathfrak{m}$-primary ideal $I$ as metric invariants of another surface explicitly constructed using a system of generators of $I$.

	   \begin{theorem}[Theorem \ref{immersion}]\label{thmz}
	Let $(X,0)$ be a complex surface germ with an isolated singularity and $I$ be a  $\mathfrak{m}$-primary ideal of $\mathcal{O}_{X,0}$. There exists a system of generators $F=(f_1,f_2,\dots,f_k)$ of $I$ such that the holomorphic map		
	$$\begin{array}{rcl}
F: (X,0) &\to& (F(X),0) \subset (\mathbb{C}^k,0) \\
p &\mapsto &(f_1(p),f_2(p),\dots,f_k(p))
\end{array} $$
verifies the following properties:
\begin{enumerate}
\item The image $(F(X),0)$ is a complex surface germ with an isolated singularity at the origin of $\mathbb{C}^k$. \label{point111}
\item The map $F$ is a homeomorphism on its image and a modification of $(F(X),0)$. \label{point222}
\item The induced map $\tilde{F}:(\mathrm{NL}(X,0), \mathrm{d}_I) \longrightarrow (\mathrm{NL}(F(X),0),\mathrm{d}_{\mathfrak{m}_F}), \tilde{F}(v)(h)=v(h \circ F), \forall v \in \mathrm{NL}(X,0), \forall h \in \mathcal{O}_{F(X),0}$ is an isometry which 	makes the following  diagram commute 
	$$ \xymatrix{
   ( \mathrm{NL}(X,0), \mathrm{d}_I) \ar[r]^{\tilde{F}}  \ar[rd]^{\mathcal{I}_{I}} & (\mathrm{NL}(F(X),0), \mathrm{d}_{\mathfrak{m}_F} ) \ar[d]^{\mathcal{I}_{\mathfrak{m}_{F}}} \\
      & \mathbb{R}_{+}^* \cup \infty
  },	$$ where $\mathfrak{m}_F$ is the maximal ideal of $\mathcal{O}_{F(X),0}$. \label{333}
    \end{enumerate}
   
  \end{theorem} 
  
Another question that naturally arises is: given two $\mathfrak{m}$-primary ideals $I$ and $J$ of $\mathcal{O}_{X,0}$, at what condition do we have the equality $\mathcal{I}_I = \mathcal{I}_J$? We provide a partial answer by proving that if $\mathcal{I}_I = \mathcal{I}_J$ then $I$ and $J$ have the same integral closure (Proposition \ref{integralclosure2}).

Now, let us briefly outline the proof of Theorem \ref{thmM} with all the tools we have introduced. The proof proceeds as follows: For each integrally closed $\mathfrak{m}$-primary ideal $I$ of $\mathcal{O}_{X,0}$, we construct a complex surface germ with an isolated singularity $(X_I,0)$ using Theorem \ref{thmz}. We then demonstrate that the function $\mathcal{I}_I$ is determined and determines the triple $(\Gamma_{\pi_I}, L_{\pi_I}, P_{\pi_I})$ introduced just before the statement of Theorem \ref{thmM}. The key idea to establish this result relies on the use of the inner rate formula \ref{laplacien}, initially proven in \cite{BFP} and generalized in \cite{yenni1}. This formula provides us that the vector $P_{\pi_I}$ and the inner rates $(q_v^I)_{v \in V(\Gamma_{\pi_I})}$ are equivalent data when given the vector $L_{\pi_I}$. Finally, by invoking Proposition \ref{integralclosure2} mentioned right before this paragraph, we conclude that if $I \neq J$, then $(\Gamma_{\pi_I}, L_{\pi_I}, P_{\pi_I}) \neq (\Gamma_{\pi_J}, L_{\pi_J}, P_{\pi_J})$.

	 \subsection*{Acknowledgments}

I would like to express my deep gratitude to my PhD advisors, André Belotto and Anne Pichon, for their invaluable help and enthusiastic encouragement during the preparation of this paper. I also wish to thank Patrick Popescu-Pampu and Evelia García Barroso for reviewing my thesis and therefore significantly improving this paper. I would also like to thank Patrick Popescu-Pampu for suggesting the idea of working with the $\mathfrak{m}$-primary ideal of the local ring. Additionally, I extend my thanks to Jawad Snoussi for the very fruitful discussions about the potential relationship between Lipschitz geometry and normalization during my stay in Mexico. This work has been supported by the \textit{Centre National de la Recherche Scientifique (CNRS)}, which funded my PhD scholarship.

\section{Preliminaries}
		This section is dedicated to reviewing the necessary material for this paper. The definitions and results presented in Subsections \ref{subsection1} and \ref{subsection2} can be found in \cite{BFP, GignacRuggiero2017, MaugendreMichel2017, Michel2008, yenni1}, while the content of Subsections \ref{subsection3} and \ref{subsection4} is detailed in \cite{yenni1} and \cite{yennithese} .

		\subsection{Resolution of curves and surfaces}\label{subsection1}
In this subsection we introduce some classical tools about resolution of singularities of curves and surfaces. 

\begin{defi}[Resolution of singularities]
Let $(X,0)$ be a complex surface germ. A \textbf{resolution} of $(X,0)$ is a proper modification $\pi : (X_{\pi},E) \longrightarrow (X,0)$ such that $X_{\pi}$ is a smooth  complex surface  and the restricted  function $\pi_{|X_{\pi} \backslash E}: X_{\pi} \backslash E  \longrightarrow X \backslash \mathrm{Sing}(X)$  is a biholomorphism. The curve $E=\pi^{-1}(0)$ is called the \textbf{exceptional divisor}. The resolution $\pi : (X_{\pi},E) \longrightarrow (X,0)$ is   \textbf{good}  if $E$ is a simple normal crossing divisor i.e., it has smooth compact irreducible components and the singular points of $E$ are  transversal double points. Let $E_v$ be an irreducible component of $E$. A \textbf{curvette}  of $E_v$ is a  smooth curve germ which intersects  $E_v$ transversely at a smooth point of $E$.   \end{defi}
\begin{defi}[Strict transform and total transform of a curve germ]
Let $\pi : (X_{\pi},E) \longrightarrow (X,0)$ be a modification of $(X,0)$ and let $(C,0)$ be a curve germ  in $(X,0)$. The \textbf{strict transform} of $C$ by $\pi$ is the curve  $C^*$  in $X_{\pi}$ defined as the topological closure of the set $\pi^{-1}(C \backslash \{0\})$. Let $E_1,E_2,\ldots,E_n$ be the irreducible components of $E$ and $h:(X,0) \longrightarrow (\mathbb{C},0) $ be a holomorphic function.  The \textbf{total transform} of $h$ by $\pi$  is the principal divisor $(h \circ \pi)$ on $X_\pi$, i.e., 
$$ (h \circ \pi)= \sum_{i=1}^n m_i(h) E_i+h^*$$
where $m_i(h)$ is the order of vanishing  of the holomorphic function $h  \circ \pi $ on the irreducible component $E_i$ of $E$ and $h^*$ is the strict transform of the curve $h^{-1}(0)$.
\end{defi}




\begin{defi}[Good resolution of curve germ]
Let $(C,0)$ be a curve germ in $(X,0)$.  A proper modification $\pi : (X_{\pi},E) \longrightarrow (X,0)$ is  a \textbf{good resolution of $(X,0)$ and $(C,0)$} if it is a  good resolution of  $(X,0)$ such that
the strict transform $C^*$ is a disjoint union of curvettes.
\end{defi}

\begin{defi}[Dual graph of a resolution]
The \textbf{dual graph} of  a good resolution $\pi :(X_\pi,E) \longrightarrow (X,0) $ of $(X,0)$ is the graph $\Gamma_\pi$ whose vertices are in bijection with the irreducible  components of $E$ and  such that the  edges between  the vertices $v$ and $v'$  corresponding to $E_v$ and $E_{v'}$ are in bijection with   $E_v \cap E_{v'}$, each vertex $v$ of this graph is weighted with the self intersection number $E_v^2$ and the genus $g_v$ of the corresponding curve $E_v$. We denote by $V(\Gamma_{\pi})$ the set of vertices of $\Gamma_{\pi}$ and $E(\Gamma_{\pi})$ the set of edges. The \textbf{valency of a vertex $v$} is the number  $\mathrm{val}_{\Gamma_{\pi}}(v):=\left(\sum_{i \in V(\Gamma_{\pi}),  i \neq v }E_i  \right)\cdot E_v$.


\end{defi}

   \subsection{Non-archimedean link}\label{subsection2}
   
The following definitions and results can be found in \cite[Preliminaries]{BFP} and \cite[Section 2]{GignacRuggiero2017}.
Let $(X,0)$ be a complex surface germ with an isolated singularity. Denote by $\mathcal{O}=\widehat{\mathcal{O}_{X,0}}$ the completion of the local ring of $X$ at $0$ with respect to its maximal ideal.
\begin{defi}[Semivaluation]
A (rank $1$) \textbf{semivaluation} on $\mathcal{O}$ is a map $v:\mathcal{O} \longrightarrow \mathbb{R} \cup \{ +\infty \} $ such that, for every $f,g \in \mathcal{O}$ and every $\lambda \in \mathbb{C}^{*}$

\begin{itemize}

\item $v(fg)=v(f)+v(g)$
\item$ v(f+g) \geq \mathrm{min}\{v(f),v(g)\}$
\item $v(\lambda)=\left\{\begin{array}{@{}l@{}}
    + \infty \ \text{if} \ \lambda=0\\
    
     0 \ \ \ \ \text{if} \ \lambda \neq 0
  \end{array}\right.\,$.
\end{itemize}  A \textbf{valuation} is a  semivalutation such that $0$ is the only element sent to $+\infty$. 
\end{defi}

\begin{defi}[Evaluation of an ideal]
Given an ideal $I$ of $\mathcal{O}$ and a semivaluation $v$ we define the evaluation of $I$ by $v$ to be $$ v(I)=\mathrm{inf}\{v(h) \ | \ h \in I\}.$$
\end{defi}

\begin{exam}[Normalized divisorial valuations]\label{divisorial}  Let $\pi: (X_\pi,E) \longrightarrow (X,0)$ be a good resolution of $(X,0)$. Let $E_v$ be an irreducible component of the exceptional divisor $E$. Let $\mathfrak{m}$ be the maximal ideal of $\mathcal{O}$ and set $m_v(\mathfrak{m}):=\mathrm{inf}\{ m_v(h) \ | \ h \in \mathfrak{m}\}$. The map

 $$\begin{array}{rcl}
\mathrm{val}_{E_v}:\mathcal{O}&\to& \mathbb{R}_{+} \cup \{+ \infty \},\\
f &\mapsto &\frac{m_v(f)}{m_v(\mathfrak{m})}
\end{array} $$   is a valuation on $\mathcal{O}$. We call it the \textbf{divisorial valuation} associated with $E_v$.
\end{exam}

\begin{exam}[Normalized monomial valuations]\label{monomial}  Let $\pi: (X_\pi,E) \longrightarrow (X,0)$ be a good resolution of $(X,0)$. Let $p \in X_{\pi}$ be the intersecting point of two irreducible component $E_v$ and $E_{w}$ of the exceptional divisor $E$. Let $(x,y)$ be a local coordinate system centered at $p$ such that $x=0$ and $y=0$ are respectively the equations of $E_v$ and $E_w$. Then, for every element  $f$ of  $\mathcal{O}$ we have $f\circ \pi (x,y) = \sum_{ij}a_{ij}x^{i}y^{j} \in \mathbb{C} \llbracket x,y \rrbracket$. For every $t \in [0,1]$ we consider the valuation $$ \mu_t(f)=\min \left\{\frac{1-t}{m_v(\mathfrak{m})}i + \frac{t}{m_w(\mathfrak{m})}j \ |\ a_{ij} \neq 0 \right\}$$ called \textbf{Normalized monomial valuation} associated with $E_v$ and $E_w$ with weight $t$.  
Notice that $\mu_0= \mathrm{val}_{E_v}$ and $\mu_1= \mathrm{val}_{E_w}$. Then we denote by $[\mathrm{val}_{E_v}, \mathrm{val}_{E_w}]$ the set of the normalized monomial valuations $\mu_t$ for $t \in [0,1]$.

 \end{exam}

\begin{lemm}[{\cite[Subsection 2.1]{GignacRuggiero2017}}]\label{density}
With the same notation as in Example \ref{monomial}. If $t \in [0,1]$ is a rational number then the semivaluation $\mu_t$ is a divisorial valuation.

\end{lemm}

\begin{defi}[Non-archimedean link]\label{linknonarchimedien}
The \textbf{non-archimedean link} $\mathrm{NL}(X,0)$ of $(X,0)$ is the topological set $$ \mathrm{NL}(X,0) = \{ v: \mathcal{O} \longrightarrow \mathbb{R}_{+} \cup \{+\infty\}  \ \text{semi-valuation} \ | \ v(\mathfrak{m}) =1 \ \text{and} \ v_{|\mathbb{C}^* } =0  \} $$
whose topology is induced by the product topology $(\mathbb{R}_{+} \cup \{+ \infty \})^{\mathcal{O}}$.
\end{defi}
Let $\pi: (X_\pi,E) \longrightarrow (X,0)$ be a good resolution of $(X,0)$. There exists an embedding 
$$i_{\pi} :\Gamma_{\pi} \longrightarrow \mathrm{NL}(X,0)$$ where $\Gamma_{\pi}$ is endowed with the topology such that each of its edge is homeomorphic to a real interval. There exists a continuous retraction $$r_{\pi} :\mathrm{NL}(X,0) \longrightarrow \Gamma_{\pi}$$ such that $r_{\pi} \circ i_{\pi} = \mathrm{Id}_{\Gamma_{\pi}}$. The embedding $i_{\pi}$ maps each vertex $v$ of $\Gamma_{\pi}$ to the normalized divisorial valuation associated with the component $E_v$, and each edge $e_{v,v'}$ that corresponds to a point $p$ of the intersection $E_v \cap E_{v'}$ to the set of normalized monomial valuations on $X_{\pi}$ at $p$. In this paper we will mostly make use of the following theorem which gives a description of the non-achimedean link in terms of dual graphs.

 \begin{thm}[{See e.g., \cite[Theorem 2.27]{GignacRuggiero2017} and \cite[ Theorem 7.9]{Jonsson2015}}] \label{universaldualgraph}
 The family of the continuous retractions $\{ r_{\pi} \ | \ \pi \ \text{is a good resolution of} \ (X,0) \}$ induces a  homeomorphism from $\mathrm{NL}(X,0)$ to the inverse limit of the dual graphs $\Gamma_{\pi}$
$$\mathrm{NL}(X,0) \cong \varprojlim_{ \pi } \Gamma_{\pi}. $$
 \end{thm} 
\begin{rema}\label{densiteremarque}
	It follows from Lemma \ref{density} that given a good resolution $\pi:(X_{\pi},E)\longrightarrow (X,0)$ and two irreducible components $E_{v}$ and $E_{w}$ of $E$  the set of divisorial valuations of $[\mathrm{val}_{E_v},\mathrm{val}_{E_w}]$ is dense in $[\mathrm{val}_{E_v},\mathrm{val}_{E_w}]$.\end{rema}
 Let $F: (X,0) \longrightarrow (Y,0)$ be a finite holomorphic map between two complex surface germ $(X,0)$ and $(Y,0)$. It induces a continuous map $\widetilde{F} : \mathrm{NL}(X,0) \longrightarrow \mathrm{NL}(Y,0)$. Indeed, we set
$F^{\#}:\widehat{ \mathcal{O}_{(Y,0)}} \longrightarrow \widehat{ \mathcal{O}_{(X,0)}}$ defined by  $F^{\#} (h)=h \circ F$. Hence

$$  \begin{array}{rcl}
\widetilde{F}:\mathrm{NL}(X,0)&\to&  \mathrm{NL}(Y,0).\\
v &\mapsto & \frac{v \circ F^{\#}}{v(F^*(\mathfrak{m}_Y))}
\end{array} $$ where $\mathfrak{m}_Y$ is the maximal ideal of the completion $\widehat{\mathcal{O}_{Y,0}}$. The map $\tilde{F}$ has the following properties
\begin{prop}[{\cite[Proposition 4.4]{GignacRuggiero2017}}]\label{onetoonearchimede} 	  The following holds for any finite holomorphic map $F:(X,0) \longrightarrow (Y,0)$
	\begin{enumerate}
		\item the map $\tilde{F}$ is surjective.
		\item every $v \in \mathrm{NL}(Y,0)$ has at most $N$ preimages under $\tilde{F}$, where $N$ is the degree of the fields extension $\frac{\mathrm{Frac}(\widehat{\mathcal{O}_X})}{F^*\mathrm{Frac}(\widehat{\mathcal{O}_Y})}$. In particular if $F$ is a finite modification then $\tilde{F}$ is one to one.	\end{enumerate}
\end{prop}

    \subsection{Inner-rates of a finite morphism}\label{subsection3}
   
Let $(X,0)$ be a complex surface germ with an isolated singularity, and let $g, f : (X,0) \longrightarrow (\mathbb{C},0)$ be two holomorphic functions such that the morphism $\Phi = (g,f) : (X,0) \longrightarrow (\mathbb{C}^2,0)$ is finite. The aim of this subsection is to recall the notion of inner rates associated with a finite morphism $\Phi = (g,f) : (X,0) \longrightarrow (\mathbb{C}^2,0)$ as defined in \cite[Definition 2.2]{yennithese}, which generalizes the \textit{inner rates of complex surface germs} introduced in \cite[Definition 3.3]{BFP}. We will provide the proof of the necessary Proposition \ref{inner-rate} without additional comments. For more context and examples, we encourage readers to refer to the papers mentioned above.



\begin{defi} [{See e.g \cite[Definition 0.1]{KTP}}]\label{polar curve}  The \textbf{polar curve} of the morphism $\Phi $ is the curve $\Pi_{\Phi}$  defined as the topological closure of the critical locus of the finite morphism $\Phi=(g,f)$, that is
$$ \Pi_{\Phi} = \overline{  \{   x \in  X \backslash \{0\} \ | \ d_x \Phi : T_x X \longrightarrow \mathbb{C}^2 \ \text{ is non surjective} \}   }.$$ Equivalently it can be defined as the vanishing locus of the holomorphic $2$-form $\mathrm{d}g \wedge \mathrm{d}f$.
\end{defi}


\begin{prop}[{\cite[Proposition 2.3]{yennithese}, \cite[Proposition 2.4]{yenni1}}]\label{inner-rate}
Let $\pi :(X_{\pi},E) \longrightarrow (X,0)$ be a good resolution of $(X,0)$ and let $E_v$ be an irreducible component of the exceptional divisor $E$. Let us denote by $(u_1,u_2)=(g,f)$ the coordinates of $\mathbb{C}^2$ and by $\mathrm{d}$ the standard hermitian metric of $\mathbb{C}^2$. 
 Then, for every smooth point $p$ of $E$  in $E_v \backslash (f^* \cup g^* \cup \Pi_{\Phi}^*)$, there exists an open neighborhood  $O_p \subset E_v$ of $p$ such that for every pair of curvettes $\gamma_{1}^*, \gamma_{2}^*$ of $E_v$ verifying:
 
\setcounter{equation}{0}
\renewcommand{\theequation}{2.\arabic{equation}}\begin{equation}\label{cond}
\left\{
\begin{array}{rcl}
 \gamma_1^* \cap \gamma_2^*&=&\emptyset \\
\gamma_{i}^* \cap O_p &\neq& \emptyset \ \text{for} \ i=1,2.
\end{array}
\right.
\end{equation}
we have: $$ \mathrm{d}(\gamma_1 \cap \{u_2 = \epsilon \}, \gamma_2 \cap \{u_2= \epsilon \} ) =\Theta(\epsilon^{q_{g,v}^f}), \ q_{g,v}^f:=\frac{\mathrm{Ord}_{E_v}(\pi^{*}(\mathrm{d}g \wedge \mathrm{d}f))-m_v(f)+1}{m_v(f)}$$ where  $\gamma_1=(\Phi \circ \pi)(\gamma_1^*), \gamma_2=(\Phi \circ \pi) (\gamma_2^*)$ and $\epsilon \in \mathbb{R}$.  \end{prop}
\begin{defi}\cite[Definition 2.3]{yenni1}\label{inner}
We call  $q_{g,v}^f$ the \textbf{inner rate of $f$ with respect to $g$ along $E_v$}.
\end{defi}
\begin{proof}[Proof of Proposition \ref{inner-rate}]
Let $p$ be a smooth  point of $E_v$ which does not belong to  the strict transforms $f^*,g^*$ and $\Pi_{\Phi}^*$. Let $(x,y)$ be a local system of coordinates of $X_{\pi}$ centered at $p$ such that $E_v$ has  local  equation $x=0$ and such that $(f \circ \pi)(x,y)=x^{m_v(f)}$ .
Let  $U$ be the  unit of $\mathbb{C}\{ x,y\}$ such that $ (g \circ \pi)(x,y) = x^{m_v(g)}U(x,y)$. We can write 
$$ U(x,y)=\sum_{i \geq 0} a_{i0}x^i  + \sum_{j \geq 1}y^j \sum_{i \geq 0} a_{ij} x^i$$
Since $\Phi$ is a  finite morphism  the set $\{ i \geq 0 \ | \ \exists j > 0, a_{ij} \neq 0  \}.$ is non empty. Let $k$ be its minimal element. Then,
$$ U(x,y)=g_{0}(x)  +x^k \sum_{j \geq 1}y^j g_{j}(x),$$
where  $g_{0}(x)=\sum_{i \geq 0} a_{i0}x^i$ and $g_{j}(x) =\sum_{i \geq 0} a_{ij} x^{i-k}$. Note that the set $\{ j > 0 \ | \ g_{j}(0) \neq 0  \}$ is non empty.

Setting $q_p:=\frac{m_v(g)+k}{m_v(g)}$, we then have:
\begin{eqnarray}\label{forme}
(g \circ \pi)(x,y) = x^{m_v(g)}g_{0}(x)+x^{q_pm_v(g)}\sum_{j \geq 1} y^j g_{j}(x).   \end{eqnarray}

\noindent Let $\gamma_{1}^*$ and $\gamma_{2}^*$ be two curvettes of $E_v$ parametrized respectively by
$$ t \mapsto (t,\alpha + th_1(t)), \ t \mapsto (t,\beta + th_2(t) ), \ \alpha,\beta \in \mathbb{C} $$
where $h_1$ and $h_2$ are convergent power series. The curves $\gamma_1= (\Phi \circ \pi)(\gamma_1^*) $ and $\gamma_2=(\Phi \circ \pi)(\gamma_2^*)$  are then parametrized respectively by
$$  t \mapsto (t^{m_v(g)}g_{0}(t)+t^{q_pm_v(g)}\sum_{j \geq1}(\alpha+
th_1(t))^j g_{j}(t) ,t^{m_v(f)})$$ and
$$  t \mapsto (t^{m_v(g)}g_{0}(t)+t^{q_pm_v(g)}\sum_{j \geq1}(\beta+
th_2(t))^j g_{j}(t) ,t^{m_v(f)}).$$
Therefore, for $\epsilon > 0$, we have $$\mathrm{d}(\gamma_1 \cap \{ u_2 = \epsilon \}, \gamma_2 \cap \{ u_2 =\epsilon\}) 
=\displaystyle\left\lvert \epsilon^{\frac{q_p m_v(g)}{m_v(f)}} \displaystyle\right\rvert H(\epsilon), $$ where 
$$H(\epsilon)=\displaystyle\left\lvert\sum_{j \geq 1 }\left( (\alpha+ \epsilon^{\frac{1}{m_v(f)}}h_1(\epsilon^{\frac{1}{m_v(f)}}))^j - (\beta + \epsilon^{\frac{1}{m_v(f)}}h_2(\epsilon^{\frac{1}{m_v(f)}}))^j \right) g_{j}(\epsilon^{\frac{1}{m_v(f)} } ) \displaystyle\right\rvert.$$ We need to prove that  $$ H(0)=\displaystyle\left\lvert \sum_{j \geq 1 }\left( \alpha^j - \beta^j \right) g_{j}(0) \displaystyle\right\rvert$$ does not vanish when $\alpha$ and $\beta$ are distinct and in a small enough neighborhood of the origin of $\mathbb{C}$.
Let $j_0 > 0$ be the minimal element of the set $\{ j > 0 \ | \ g_{j}(0) \neq 0  \}$ then 
$$ H(0)=\displaystyle\left\lvert (\alpha^{j_0}-\beta^{j_0})g_{j_0}(0) + \sum_{j \geq 1, j \neq j_0 }\left( \alpha^j - \beta^j \right) g_{j}(0) \displaystyle\right\rvert.$$ Now, let us prove  that $j_0=1$. Let us compute in the coordinates $(x,y)$ the equation of the total transform of the polar curve.
 The jacobian matrix of $ \Phi \circ \pi$ is  $$\mathrm{Jac}( \Phi \circ \pi)(x,y) = \begin{pmatrix}
* & x^{q_{p} m_v(g)}(g_{1}(x)+2yg_{2}(x)+\ldots) \\
m_v(f) x^{m_v(f)-1} &0
\end{pmatrix}$$then $\mathrm{Det}(\mathrm{Jac}( \Phi \circ \pi)(x,y))=m_v(f) x^{q_{p} m_v(g)+m_v(f)-1}(g_{1}(x)+2yg_{2}(x)+\ldots)=0$ is the equation of the total transform of $\Pi_{\Phi}$. Since $g_{j_0}(0) \neq 0$ it follows that the equation of the  strict transform of the polar curve $\Pi_{\Phi}^*$ is    $$\sum_{j \geq 1 }  jy^{j-1} g_{j}(x) =0.$$
By hypothesis $p \notin \Pi_{\Phi}^*$ which means that $g_{1}(0) \neq 0$ and then $j_0 =1$. Thus
$$ H(0)=|(\alpha-\beta)g_{1}(0) + \sum_{j > 1}\left( \alpha^j - \beta^j \right) g_{j}(0) |.$$  If $\alpha$ and $\beta$ are distinct and in a small enough neighborhood of $0 \in \mathbb{C}$ we have $H(0) \neq 0$. With this, we can conclude that there exists a neighborhood $O_p$ of $p$ such that for every curvettes satisfying \eqref{cond}
$$\mathrm{d}(\gamma_1 \cap \{ u_2 = \epsilon \}, \gamma_2 \cap \{ u_2 =\epsilon\}) =\displaystyle\left\lvert \epsilon^{ \frac{q_p m_v(g)}{m_v(f)} } \displaystyle\right\rvert H(\epsilon)= \Theta( \epsilon^{\frac{m_v(g)q_p}{m_v(f)}} ).$$ Now we will see that the number $\frac{q_pm_v(g)}{m_v(f)}$ does not depend of the point $p$. 

Let us consider the holomorphic 2-form $\Omega=\pi^*(\mathrm{d}g \wedge \mathrm{d}f)$ on $X_{\pi}$. In the neighborhood of $p$, a direct computation shows that the holomorphic $2$-form $\Omega$ is given in the coordinates $(x,y)$ by: $$
\Omega=-m_v(f)x^{q_p m_v(g)+m_v(f)-1}\sum_{j \geq 1 }  jy^{j-1} g_{j}(x) \mathrm{d}x \wedge \mathrm{d}y.
$$ It follows that the order of vanishing of the holomorphic $2$-form $\Omega$  along the irreducible component $E_v$ is 
$$ \mathrm{Ord}_{E_v}(\pi^*(\mathrm{d}g \wedge \mathrm{d}f))=q_p m_v(g)+m_v(f)-1.$$ The number $\frac{q_p m_v(g)}{m_v(f)}$ is then independant of the point $p$ and is equal to

$$ \frac{q_p m_v(g)}{m_v(f)}=q_{g,v}^f:=\frac{\mathrm{Ord}_{E_v}(\pi^{*}(\mathrm{d}g \wedge \mathrm{d}f))-m_v(f)+1}{m_v(f)}.$$

 \end{proof}

 \subsection{Inner rates formula}\label{subsection4}

Let us recall the inner rates formula from \cite[Theorem 3.1]{yenni1}, which generalizes a previous result from \cite[Theorem 4.3]{BFP}. This formula will play a key role in the proof of Theorem \ref{thmM} mentioned in the introduction. In this subsection, we will provide a proof of this formula to ensure that the paper is as self-contained as possible. For those seeking more details, we encourage you to refer to the previously mentioned papers.

  \begin{thm}[{The inner rates formula}] \label{laplacien} Let $(X,0)$ be a complex surface germ with an isolated singularity and let $\pi :(X_{\pi},E) \longrightarrow (X,0) $ be a good resolution of $(X,0)$. Let $g,f:(X,0) \longrightarrow (\mathbb{C},0)$ be two holomorphic functions on $X$ such that the morphism $\Phi=(g,f): (X,0) \longrightarrow (\mathbb{C}^2,0)$ is finite. Let $M_{\pi}=(E_{v_i} \cdot E_{v_j})_{i,j \in \{1,2,\ldots,n\}}$ be the \textbf{intersection matrix} of the dual graph $\Gamma_{\pi}$, $a_{g,\pi}^f:=(m_{v_1}{(f)}q_{g, v_1}^f,\ldots,m_{v_n}(f)q_{g, v_n}^f)$, $K_{\pi} :=( \mathrm{val}_{\Gamma_{\pi}} (v_1) +2g_{v_1}-2,\ldots,\mathrm{val}_{\Gamma_{\pi}} (v_n) +2g_{v_n}-2)$, $F_{\pi}=(f^* \cdot E_{v_1},\ldots,f^* \cdot E_{v_n} )$ be \textbf{the $F$-vector } and 
$P_{\pi}=(\Pi_{\Phi}^* \cdot E_{v_1},\ldots,\Pi_{\Phi}^* \cdot E_{v_n})$ be \textbf{the polar vector or $\mathcal{P}$-vector}. Then we have:
$$M_{\pi}  .\underline{a_{g,\pi}^f}=\underline{K_{\pi}}+\underline{F_{\pi}}-\underline{P_{\pi}}.$$
Equivalently, for each irreducible component $E_v$ of $E$ we have the  following:
$$
 \left( \sum_{i \in V(\Gamma_{\pi})} m_{i}(f)q_{g,i}^f E_i  \right) \cdot E_{v}= \mathrm{val}_{\Gamma_{\pi}}(v)+f^* \cdot E_v-\Pi_{\Phi}^* \cdot E_v+2g_v-2.
$$
\end{thm}
Let us introduce the required tool in order to prove Theorem \ref{laplacien}.
\begin{defi}[{See e.g \cite[Chapter 3. Subsection 6.3]{shafarevich}}]
Let $S$ be a smooth complex surface. The \textbf{canonical divisor} of $S$, denoted $K_S$, is the  divisor $ K_{\Omega}$ associated to any meromorphic $2$-form $\Omega$ defined on $S$. It is  well defined up to linear equivalence i.e.,  for any pair of $2$-forms $\Omega$ and $\Omega'$ of $S$ there exists a meromorphic function $h$ on $S$  such that $$ K_{\Omega}=(h)+K_{\Omega'}$$ where  $K_{\Omega}$ and $K_{\Omega'}$ are respectively the divisors associated to the $2$-forms $\Omega$ and $\Omega'$.
\end{defi}

\begin{thm}[{Adjunction formula, see e.g \cite[Chapter 4. Subsection 2.3]{shafarevich}}] \label{adj} Let $S$ be a complex surface and $C \subset S$ be  a compact Riemann surface embedded in $S$. Then 
$$ (K_S+C) \cdot C=2g_{C}-2, \ \ \  \text{where} \ g_{C}  \ \text{is the genus of } \ C.  $$

\end{thm}

\begin{proof}[Proof of Theorem \ref{laplacien}]

Let us consider the holomorphic $2$-form $\Omega:=\pi^*(\mathrm{d}g \wedge \mathrm{d}f)$ defined on the smooth complex surface $X_{\pi}$. Since $\pi_{X_{\pi} \backslash E}:X_{\pi} \backslash E \longrightarrow X \backslash 0$ is an isomorphism, the vanishing locus of $\Omega$ is $E \cup \Pi_{\Phi}^*$. The associated canonical divisor of $X_{\pi}$ is given by
 
$$ K_{\Omega}= \sum_{v_i \in V(\Gamma_{\pi})} \mathrm{Ord}_{E_{v_i}}(\pi^*(\mathrm{d}g \wedge \mathrm{d}f))E_{v_i}+ \Pi_{\Phi}^*. $$ By definition of the inner rates   we can rewrite the divisor $K_{\Omega}$ as follows
\begin{equation}\label{eqdemerde}K_{\Omega}= \sum_{v_i \in V(\Gamma_{\pi})} (m_{v_i}(f)q_{g,v_i}^f + m_{v_i}(f) -1)E_{v_i}+ \Pi_{\Phi}^* .\end{equation}  We now apply the adjunction formula \ref{adj} to the compact Riemann surface  $E_v$  \begin{eqnarray}\label{adjunction}
(K_{\Omega}+E_v) \cdot E_v=2g_v-2.\end{eqnarray} Combining equations (\ref{eqdemerde}) and (\ref{adjunction}), we obtain:  \begin{eqnarray}\label{eq1996} \sum_{v_i \in V(\Gamma_{\pi})} q_{g,v_i}^{f}m_{v_i}(f) E_{v_i} \cdot E_v+ \sum_{v_i \in V(\Gamma_{\pi} )} m_{v_i}(f)E_{v_i}\cdot E_v -\sum_{v_i \in V(\Gamma_{\pi}), v_i \neq v }E_{v_i} \cdot E_v +\Pi_{\Phi}^* \cdot E_v=2g_v-2.\end{eqnarray} Recall that we denote by $(f \circ \pi)$ the principal divisor associated with the holomorphic function $f \circ \pi$.

 Finally, we have the following formula which comes from the fact that $(f \circ \pi)\cdot E_v=0$: $$\left( \sum_{v_i \in V(\Gamma_{\pi} )} m_{v_i}(f)E_{v_i}\cdot E_v \right)=\left(-f^* \cdot E_v\right).$$ Replacing this in Equation \ref{eq1996}, we get the desired equality: $$ \sum_{v_i \in V(\Gamma_{\pi})} q_{g,v_i}^{f}m_{v_i}(f) E_{v_i} \cdot E_v= \mathrm{val}_{\Gamma_{\pi}}(v) +f^* \cdot E_v-\Pi_{\Phi}^* \cdot E_v+2g_v-2.$$\end{proof}

  \section{Inner rates of primary ideals}\label{chapter4}

\subsection{Definition of the inner rates of an m-primary ideal}\label{section4.1}
Let $(X,0)$ be a complex surface germ with an isolated singularity. Let $I$ be an $\mathfrak{m}$-primary ideal of  the local ring $\mathcal{O}_{X,0}$ associated with $(X,0)$. This means that there exists a positive natural number $n$ such that  $\mathfrak{m}^n \subset I$, where $\mathfrak{m}$ is the maximal ideal of $\mathcal{O}_{X,0}$. The aim of this section is to study a notion of inner rates with respect to the ideal $I$: \begin{defi}\label{inner-rate-idealdefi}
	Let $\pi:(X_{\pi},E) \longrightarrow (X,0)$ be a good resolution of $(X,0)$ and let $E_v$ be an irreducible component of $E$. We set	$$m_v(I):=\mathrm{inf}\{m_v(h) \ | \  h \in I \},$$ and $$\nu_v(I):= \mathrm{inf} \{\mathrm{Ord}_{E_v}(\pi^{*}(\mathrm{d}h_1 \wedge \mathrm{d}h_2) ) \ | \  h_1, h_2 \in I \}. $$
We call the rational number $$ q_v^{I}:=\frac{\nu_{v}(I)-m_v(I)+1}{m_v(I)}$$ \textbf{the inner rate of the ideal  $I$ along $E_v$}.

\end{defi}
The motivation behind this definition lies in  the existence (Proposition \ref{inner-rate-ideal}) of a suitable  system of generators $(f_1,f_2,\dots,f_k)$ for the ideal $I$ (called \textit{precomplete}) which verifies the following condition. \textit{There exists a Zariski open set $V_{I}^{\pi}$ of $ \mathbb{P}(\mathbb{C}^k) \times \mathbb{P}(\mathbb{C}^k)$  such that for every good resolution $\pi:(X_{\pi},E) \longrightarrow (X,0)$, every irreducible component $E_v$ of $E$ and every element $(\alpha =[\alpha_1: \dots:\alpha_k], \beta=[\beta_1:\dots :\beta_k])$ of $V_I^{\pi}$   we have 
$$m_v(F_{\alpha})=m_v(I)=m_v(F_{\beta})$$ and $$\mathrm{Ord}_{E_v}(\pi^{*}(\mathrm{d}F_{\alpha} \wedge \mathrm{d}F_{\beta}))=\nu_v(I),$$ where $F_{\alpha}=\sum_{i=1}^{k} \alpha_i f_i$ and $F_{\beta}=\sum_{i=1}^{k}\beta_i f_i$}. 

It follows from Proposition \ref{inner-rate} that the inner rates  $$q_{F_{\alpha},v}^{F_{\beta}}=\frac{\mathrm{Ord}_{E_v}(\pi^{*}(\mathrm{d}F_{\alpha} \wedge \mathrm{d}F_{\beta} ))-m_v(F_{\beta})+1}{m_v(F_{\beta})}$$  associated to the morphism $\Phi_{\alpha,\beta}=(F_{\alpha},F_{\beta}):(X,0) \longrightarrow (\mathbb{C}^2,0)$ does not depend on the choice of $(\alpha,\beta)$ in $V_I^{\pi}$ and is equal to $q_v^{I}$. 
 The proof  of Proposition \ref{inner-rate-ideal} will need the classical notion of  blow-up of ideals which we recall now.

   \subsection{Preliminaries on the blow-up of an ideal}\label{section4.2}
    Let $F=(f_1,f_2,...,f_k)$ be a set of generators for the $\mathfrak{m}$-primary ideal $I$. Consider the map
 $$\begin{array}{rcl}
\nu_F:X \backslash \{0\} &\to& \mathbb{P}(\mathbb{C}^k) \\
p &\mapsto &[f_1(p):f_2(p):                                                                                                                                                             \dots:f_k(p)]
\end{array}. $$ The map  $\nu_F$ is well defined because the ideal $I$ is $\mathfrak{m}$-primary, thus its generators cannot vanish simultaneously outside of $0$.   Let $$X_F:=\overline{\mathrm{Graph}(\nu_F)}=\overline{\{(p,[f_1(p):\dots:f_k(p)] )\ | \ p \in X \backslash 0 \}} \subset X \times \mathbb{P}(\mathbb{C}^k)$$ and  consider the modification 

$$\begin{array}{rcl}
\mathrm{BL}_{F}:(X_F, E_F) &\to& (X,0),\\
(p,v) &\mapsto & p 
\end{array} $$
where $E_F=\mathrm{BL}_{F}^{-1}(0)$. The modification $\mathrm{BL}_{F}$ does not depend on the choice of the generators $F=(f_1,\dots,f_k)$ in the following sense:

\begin{prop}
	Let $G=(g_1,g_2,\dots ,g_s)$ be another set of generators for $I$. There exists a unique isomorphism $h:(X_G,E_G)\longrightarrow(X_F,E_F)$ such that the following diagram commutes:
$$\xymatrix{
    X_F  \ar[d]^{\mathrm{BL}_F}  & X_G \ar[l] \ar[ld]^{\mathrm{BL}_G} \\
     X
  }$$ 
  
\end{prop}\begin{defi}\label{blowupoideal}
The map $\mathrm{BL}_F$ is denoted $\mathrm{BL}_I:(X_I,E_I) \longrightarrow (X,0)$ and is called the \textbf{blow-up of the ideal $I$}.
	
\end{defi}

\noindent Let us recall the universal property of the blow-up of an ideal.

\begin{prop}[{\cite{HIRONAK64a} or see \cite[Proposition 10.2.2]{phamteissier}}]\label{hironakablowupideal}
	The blow-up $\mathrm{BL}_{I}:(X_{I},E_I)\longrightarrow (X,0)$ of the ideal  $I$ is the unique modification such that:
	\begin{enumerate}
		\item The ideal sheaf $I.\mathcal{O}_{X_{I}}:=\mathrm{BL}_{I}^{*}(I) \otimes \mathcal{O}_{X_{I}}$ is locally principal. \label{locallyfree}\label{hironakasheaf}
		\item For every modification $\phi:T \longrightarrow X$ such that $I.\mathcal{O}_{T}$ is locally principal, there is a unique factorization $r:T \longrightarrow X_{I}$ such that $\phi= \mathrm{BL}_{I} \circ r$.
		 \end{enumerate}
\end{prop} 

\noindent  Throughout this paper  we will adopt the following notations.

\begin{nota}

Let  $F=(f_1, \dots , f_k)$ be a system of generators of the ideal $I$ and $\alpha:=[\alpha_1,\alpha_2,\dots,\alpha_k] \in \mathbb{P}(\mathbb{C}^k)$. We denote by $F_{\alpha}$ the function $ \sum_{i=1}^k \alpha_i f_i$. Let $\Omega_F$ be the Zariski open set of  $\mathbb{P}(\mathbb{C}^k)$ such that  $F_{\alpha}^{-1}(0)$ is a curve germ. We denote by $L_F$ the family of curve germs $\{F_{\alpha}^{-1}(0) \}_{\alpha \in \Omega_F}.$
\end{nota}

Let us recall the notion of basepoints of a family of curves with respect to a modification.

\begin{defi}\label{pointbasedefi1996}
Let $C=\{{C_\alpha}\}_{\alpha \in \Omega}$ be a family of curves on $(X,0)$ parametrized by an open Zariski set $\Omega$ of some projective space. Let $\sigma:(Y,Z) \longrightarrow (X,0)$ be a modification. We say that a point $p$ in $Z$ is a \textbf{basepoint} for the family $C$ if $p \in C_{\alpha}^*$ for every $\alpha \in \Omega'$, where $\Omega'$ is an open Zariski set of $\Omega$ and $C_{\alpha}^*$ is the strict transform of $C_{\alpha}$ by $\sigma$.
\end{defi}

\noindent In the following proposition we give an equivalent property to Point \ref{locallyfree} of Proposition \ref{hironakablowupideal} in terms of the basepoints of a family of curves.
\begin{prop}\label{basepointofidealvssheaf}
	Let  $\sigma:(Y,Z)\longrightarrow (X,0)$ be a modification and $I$ be a $\mathfrak{m}$-primary ideal of $\mathcal{O}_{X,0}$. Let $F=(f_1,\dots,f_k)$ be a system of generators for $I$. Then the following assertions  are equivalent
	\begin{enumerate}
		\item The ideal sheaf $I.\mathcal{O}_Y:=\sigma^*(I) \otimes \mathcal{O}_Y$ is locally principal. \label{locallyprincipal}
		\item The modification $\sigma$ has no basepoint for the family $L_F$.\label{nobasepoint}	\end{enumerate}
\end{prop}
\begin{proof} \underline{We first prove that \ref{locallyprincipal} implies \ref{nobasepoint}}

  Let $p$ be a point of $Z$ and assume that $(I.\mathcal{O}_Y)_p$ is principal i.e, there exists a section $s_p=h \circ \sigma$ with $h \in I$ which  generates the ideal $(I.\mathcal{O}_Y)_p$. We prove that the point $p$ cannot be a basepoint for the family $L_F$. Let $h^*$ be the strict transform of $h^{-1}(0)$ by $\sigma$ and let us suppose it passes through the point $p$. It follows that every curve $F_{\alpha}^*, \ \alpha \in \mathbb{P}(\mathbb{C}^k)$, where $F_{\alpha}^{-1}(0) \in L_F$, contains the curve $h^*$ because $F_{\alpha} \circ \sigma = a_{\alpha} s_p, \ a_{\alpha} \in \mathcal{O}_{Y,p}$. It implies that $h^{-1}(0) \subset F_{\alpha}^{-1}(0), \ \forall \alpha \in \mathbb{P}(\mathbb{C}^k)$ and this is absurd because it would mean that the ideal $I$ is generated by $h$ and it contradicts the fact that $I$ is $\mathfrak{m}$-primary. It follows that the element $s_p$ vanishes only on $Z$. Since $F=(f_1,\dots,f_k)$ is a system of generators for $I$ there exists $i_p:=1,\dots,k$ such that $f_{i_p} \circ \sigma$ vanishes only on $Z$ in a small enough neighborhood of $p$. Knowing this, we can conclude the existence of a Zariski open set $O_p$ of $\mathbb{P}(\mathbb{C}^k)$ such that the functions $\alpha_{i_p}f_{i_p} \circ \sigma+ \sum_{i=1, i \neq i_p}^k \alpha_i f_i \circ \sigma$ do not vanish outside of $Z$ when restricted to a small enough neighborhood of $p$ for all $[\alpha_1,\dots,\alpha_{i_p},\dots,\alpha_k] \in O_p$. It means that $p$ is indeed not a basepoint for the family $L_F$.

\underline{Let us prove that \ref{nobasepoint} implies \ref{locallyprincipal}}.

 Let $p$ be a point of $Z$ such that the germ $(Z,p)$ is irreducible. By assumption, there exists a Zariski open set $O_p$ of $\mathbb{P}(\mathbb{C}^k)$ such that $F_{\alpha}^*$ does not pass through $p$ for all $\alpha$ in $O_p$. It follows that $F_{\alpha} \circ \sigma$ does not vanish outside of $Z$ when restricted to a small enough neighborhood of $p$. In particular, the vanishing set of the ideal  $(I.\mathcal{O}_{Y})_p$ is the germ of curve $(Z,p)$. Let $s_p$ be an element of $\mathcal{O}_{Y,p}$ which generates the vanishing ideal $I_{(Z,p)}$ of the germ $(Z,p)$.	Let $m$ be the maximal integer such that $s_p^m$ divides $f_i \circ \sigma \in (I. \mathcal{O}_Y)_p$ for all $i:=1,\dots,k$. Let $U_p$ be the Zariski open set of $\mathbb{P}(\mathbb{C}^k)$ such that $F_{\alpha} \circ \sigma=s_p^m h_{\alpha}, \  h_{\alpha} \in \mathcal{O}_{Y,p}$ and $s_p$ does not divide $h_{\alpha}$ for all $\alpha \in U_p$.  Then, for all $\alpha \in O_p \cap U_p $ we have $F_{\alpha} \circ \sigma=s_p^m u_{\alpha}$ where $u_{\alpha}$ is a unit of $\mathcal{O}_{Y,p}$ because $F_{\alpha} \circ \sigma$ do not vanish outside of $Z$. It follows that $s_p^m$ is an element of $(I.\mathcal{O}_Y)_p$ and thus generates it. Indeed, let $g \in (I.\mathcal{O}_Y)_p $	 then $g= \sum_{i+1}^k g_i f_i \circ \sigma, \ g_i \in \mathcal{O}_{Y,p}$ and $s_p^m$ divides every element $f_i \circ \sigma$.	Similar arguments work if $p$ is a point such that $(Z,p)$ is not irreducible.
		
		 %
		
		
		\end{proof}

\subsection{Precomplete systems of generators}\label{section4.3}


The aim of this subsection is to introduce the notion of precomplete  system of generators of an ideal and to prove Proposition \ref{inner-rate-ideal} announced in Subsection \ref{section4.1}. The proof needs several preliminary lemmas
\begin{lemm}\label{mI}
Let $F=(f_1,f_2,\dots ,f_k)$ be a set of generators for $I$ and  $\pi:(X_{\pi},E)\longrightarrow (X,0)$  be a good resolution of $(X,0)$. There exists a Zariski open set $U_{I}^{\pi}$ of $\mathbb{P}(\mathbb{C}^k)$ such that for any irreducible component $E_v$ of $E$ and any element $\alpha$ of $U_{I}^{\pi}$ we have

$$ m_v(F_{\alpha})=m_v(I).$$

\end{lemm}

\begin{proof}

Let $p$ be a smooth  point of $E_v$ and $(x,y)$ be a local system of coordinates of $X_{\pi}$ centered at $p$ such that $E_v$ has  local  equation $x=0$.  In this setting we have   $(f_i \circ \pi)(x,y)=x^{m_v(f_i)}\phi_i(x,y), \ 1 \leq i \leq k $, where $\phi_i$ is an element of $\mathbb{C}\{x,y\}$ such that $\phi_i(0,y) \neq 0$. For every  $\alpha$ in $\mathbb{P}(\mathbb{C}^k)$ we have
 $$ (F_{\alpha}\circ \pi) (x,y)=\sum_{i=1}^k \alpha_i x^{m_v(f_i)}\phi_i(x,y). $$
 Denote by $i_0$ the element of $\{1,\dots,k\}$ such that   $m_v(f_{i_0}):=\mathrm{min}\{m_v(f_i) \ | \ 1 \leq i \leq k \},$ then
 
  $$ (F_{\alpha}\circ \pi) (x,y)=x^{m_{v}(f_{i_0})} \left(\alpha_{i_{0}}\phi_{i_{0}}(x,y) + \sum_{i \neq i_{0}} \alpha_i x^{m_v(f_i)-m_{v}(f_{i_0}) }\phi_i(x,y) \right). $$ It follows that there exists a Zariski open set $U_{I}^{\pi}$ of $\mathbb{P}(\mathbb{C}^k)$ such that
$$ m_v(F_{\alpha})=m_v(f_{i_{0}}) , \forall \alpha \in U_{I}^{\pi}. $$
Now, let $f$ be an element of $I$ such that $m_v(f)=m_v(I)$. Since $F=(f_1,\dots,f_k)$ is a set of generators of $I$ then there exists $g_1,g_2, \dots ,g_k$ elements of $\mathcal{O}_{X,0}$ such that $$f=\sum_{i=1}^k g_i f_i.$$ In the local coordinates $(x,y)$ we have
$$ (f \circ \pi)(x,y)=x^{m_v(f_{i_0})} \left((g_{i_0}\circ \pi)(x,y) \phi_{i_0}(x,y)+ \sum_{i \neq i_{0}}g_i\circ \pi(x,y)x^{{m_v(f_{i})-m_v(f_{i_0})}}\phi_i(x,y) \right).$$It is then obvious that $m_v(f) = m_v(f_{i_0})$ and it implies that $m_v(F_{\alpha})=m_v(I)$ for every $\alpha$ in $U_I^{\pi}$. \end{proof}

Let $\Omega_X^2$ be the sheaf of $\mathcal{O}_X$-modules of holomorphic $2$-forms on $X$. Let $\Omega_I^2$ be the submodule of $\Omega_{X,0}^2$ generated by the set $\{\mathrm{d}g \wedge \mathrm{d}f \ | \ g,f \in I\}$. The following definition seems to be new:\begin{defi}\label{precompletegeneratorsystem}
	A system of generators $F=(f_1,\dots,f_k)$ for $I$ is \textbf{precomplete} if the set $$\{\mathrm{d}f_i \wedge \mathrm{d}f_j \ | \ i,j \in \{1,\dots,k \} \}$$ generates $\Omega_I^2$.
	\end{defi}
	 There exists a precomplete system of generators for $I$ because  the  $\mathcal{O}_{X,0}$-module $\Omega_{X,0}^2$ is finitely generated and noetherian, so the submodule $\Omega_I^2$ is finitely generated. In the following example we see that a system of generators of $I$ is not necessarily precomplete.
	\begin{exam}\label{precompleteexample}
		Consider the $\mathfrak{m}$-primary ideal $I=(x^2,y^2)$ of $\mathcal{O}_{\mathbb{C}^2,0}$. The set of generators $(x^2,y^2)$ is not precomplete. Indeed, let $g_1,g_2,h_1,h_2$ be arbitrary  elements of $\mathcal{O}_{\mathbb{C}^2,0}$ and consider the elements $g=x^2 g_1+y^2 g_2$ and $h=x^2 h_1+y^2 h_2$ of $I$. We have
		\begin{equation}\label{gbesoin} \mathrm{d}g \wedge \mathrm{d}h=\left(x^3(*)+xy(*)+y^3(*)\right)\mathrm{d}x \wedge \mathrm{d}y.\end{equation} In particular, take $g_1=\frac{1}{2}, h_1=y$ and $g_2=h_2=0$ then the element $\mathrm{d}g \wedge \mathrm{d}h =x^3 \mathrm{d}x \wedge \mathrm{d}y$  cannot be generated by  $\mathrm{d}(x^2) \wedge \mathrm{d}(y^2)=4xy \mathrm{d}x \wedge \mathrm{d}y$. On the other hand, the system of generators $(x^2,y^2,xy^2,x^2y)$ is  precomplete  for the ideal $I$. Indeed, by Equation (\ref{gbesoin}) the holomorphic $2$-forms
		$$\mathrm{d}(x^2) \wedge \mathrm{d}(x^2y)= 2x^3 \mathrm{d}x \wedge \mathrm{d}y, \ \ \ \ \mathrm{d}(x^2) \wedge \mathrm{d}(y^2)= 4xy \mathrm{d}x \wedge \mathrm{d}y,$$  $$ \mathrm{d}(xy^2) \wedge \mathrm{d}(y^2)= 2y^3 \mathrm{d}x \wedge \mathrm{d}y $$
	do generate the $\mathcal{O}_{\mathbb{C}^2,0}$-module $\Omega_I^{2}$.
		
		\end{exam}
	\begin{lemm}\label{nu}

Let  $F=(f_1,f_2,\dots ,f_k)$ be a precomplete system of generators of $I$ and  $\pi:(X_{\pi},E)\longrightarrow (X,0)$ be a good resolution of $(X,0)$. There exists a Zariski open set $W_{I}^{\pi}$ of $\mathbb{P}(\mathbb{C}^k) \times \mathbb{P}(\mathbb{C}^k)$ such that for any irreducible component $E_v$ of $E$ and any element $(\alpha, \beta)$ of $W_{I}^{\pi}$ we have
$$ \mathrm{Ord}_{E_v}(\pi^*(\mathrm{d}F_{\alpha} \wedge \mathrm{d}F_{\beta}))=\nu_v(I).$$

\end{lemm}

\begin{proof}
Let $p$ be a smooth  point of $E_v$ and $(x,y)$ be a local system of coordinates of $X_{\pi}$ centered at $p$ such that $E_v$ has  local  equation $x=0$.  In this setting we have for every $i,j \in \{1, \dots, k\}, i \neq j$  integers $m_v(i,j)$ such that    $$\pi^*(\mathrm{d}f_i \wedge \mathrm{d}f_j)(x,y)=x^{m_v(i,j)}\phi_{ij}(x,y) \mathrm{d}x \wedge \mathrm{d}y,$$ where $\phi_{ij}$ is an element of $\mathbb{C}\{x,y\}$ such that $\phi_{ij}(0,y) \neq 0$. For every  $\alpha,\beta$ in $\mathbb{P}(\mathbb{C}^k)$ we  have 
 $$ \pi^*(\mathrm{d}F_{\alpha} \wedge \mathrm{d}F_{\beta}) =\sum_{i\neq j} \alpha_i \beta_j x^{m_v(i,j)}\phi_{ij}(x,y)dx\wedge dy. $$
Let $i_0,j_0\in \{1,\dots,k\}$ such that  $ m_v(i_{0},j_{0})=\mathrm{min}\{m_v(i,j) \ | \ 1\leq i,j\leq k\}$. By similar arguments of the proof of Lemma \ref{mI}  we get the existence of a Zariski open set $W_{I}^{\pi}$ of $\mathbb{P}(\mathbb{C}^k) \times \mathbb{P}(\mathbb{C}^k)  $ such that
$$ \mathrm{Ord}_{E_v}(\pi^*(dF_{\alpha} \wedge dF_{\beta}))=m_v(i_{0},j_{0}),  \forall (\alpha,\beta) \in W_{I}^{\pi}. $$ Let $f,g$ be two elements of $I$ such that $$\mathrm{Ord}_{E_v}(\pi^*(df \wedge dg))=\nu_v(I). $$ By hypothesis  the set $\{\mathrm{d}f_i \wedge \mathrm{d}f_j\}_{1 \leq i,j \leq k}$ generates the $\mathcal{O}_{X,0}$-module  $\Omega_I^2$. Then, there exist elements $\{\psi_{ij} \}_{1 \leq i,j \leq k} $  of $\mathcal{O}_{X,0}$ such that $$\mathrm{d}f \wedge \mathrm{d}g=\sum_{1 \leq i,j\leq k} \psi_{ij} \mathrm{d}f_i \wedge \mathrm{d}f_j.$$ This implies that $$\pi^{*}(\mathrm{d}f \wedge \mathrm{d}g)=\sum_{1 \leq i,j\leq k} (\psi_{ij} \circ \pi) \pi^{*}(\mathrm{d}f_i \wedge \mathrm{d}f_j).$$It follows that  $m_v(i_{0},j_{0})$ is  equal to $\mathrm{Ord}_{E_v}(\pi^{*}(\mathrm{d}f \wedge \mathrm{d}g))=\nu_v(I)$ which implies that $\mathrm{Ord}_{E_v}( \pi^{*}(\mathrm{d}F_{\alpha} \wedge \mathrm{d}F_{\beta}))= \nu_{v}(I)$ for every $(\alpha, \beta)$ in $W_{I}^{\pi}$.

\end{proof} 
Until the end of this paper we use the following notations
\begin{nota}
	To a precomplete system of generators $F=(f_1,\dots,f_k)$ for $I$ we associate the family of  curves $\Pi_F:=\{\Pi_{\alpha \beta}\}_{(\alpha, \beta) \in O_F }$ where $O_F$ is the Zariski open set of $\mathbb{P}(\mathbb{C}^k) \times  \mathbb{P}(\mathbb{C}^k)$ such that the morphism $\Phi_{\alpha\beta}=(F_{\alpha},F_{\beta}):(X,0) \longrightarrow (\mathbb{C}^2,0)$ is finite for every $(\alpha,\beta)$ in $O_F$ and  $\Pi_{\alpha \beta}$ is the polar curve of the morphism  $\Phi_{\alpha,\beta}$. 	 \end{nota}

\begin{lemm}\label{basepointpolar} There exists a good resolution $\pi:(X_{\pi},E) \longrightarrow (X,0)$ such that for any precomplete system of generators $F=(f_1,f_2,\dots,f_k)$  of $I$ the modification $\pi$ has no basepoints for the family $\Pi_F$.
\end{lemm}	 \begin{defi}\label{principalization}
	We  call the minimal good resolution  as in Lemma \ref{basepointpolar} the \textbf{principalization of the $\mathcal{O}_{X,0}$-module $\Omega_I^2$}.
     \end{defi}

\begin{proof}[Proof of Lemma \ref{basepointpolar}]

Let $F=(f_1, \dots,f_k)$ be a precomplete system of generators for $I$. Let us first prove that there exists a good resolution of $(X,0)$ which has no basepoint for the family $\Pi_{F}$. Let $\pi:(X_{\pi},E)\longrightarrow(X,0)$ be a good resolution and  $E_v$ be an irreducible component of $E$. 

Let $p$ be a smooth  point of $E_v$ and $(x,y)$ be a local system of coordinates of $X_{\pi}$ centered at $p$ such that $E_v$ has  local  equation $x=0$. We use again the notations of Lemma \ref{nu}   $$\pi^*(df_i \wedge df_j)(x,y)=x^{m_v(i,j)}\phi_{ij}(x,y)dx \wedge dy,$$ where $\phi_{ij}$ is an element of $\mathbb{C}\{x,y\}$ such that $\phi_{ij}(0,y) \neq 0$. Let $\alpha$ and  $\beta$ be two elements of $\mathbb{P}(\mathbb{C}^k)$. We  have 
 $$ \pi^*(dF_{\alpha} \wedge dF_{\beta}) =\sum_{i\neq j} \alpha_i \beta_j x^{m_v(i,j)}\phi_{ij}(x,y)dx\wedge dy. $$
 By Lemma \ref{nu} there exists a Zariski open set $W_{I}^{\pi}$ such that  $$\pi^{*}(dF_{\alpha} \wedge dF_{\beta})=x^{\nu_v(I) }S_{\alpha \beta}(x,y)dx \wedge dy, \ \text{and} \ S_{\alpha \beta}(0,y) \neq 0 \ \forall (\alpha,\beta) \in W_{I}^{\pi},$$	 
where $S_{\alpha \beta}(x,y)=\sum_{i\neq j} \alpha_i \beta_j x^{m_v(i,j)-\nu_v(I)}\phi_{ij}(x,y)$.  By definition  $S_{\alpha \beta}(x,y)=0$ is the local equation of the strict transform of the curve $\Pi_{\alpha \beta}$ by $\pi$. 

Let us assume that $p$ is a basepoint for the family of curves $\Pi_F$. If there exists $(\alpha_0,\beta_0)$ in $ W_{I}^{\pi}$  such that $\Pi_{\alpha_0\beta_0}^*$ does not pass through $p$ then $S_{\alpha_0\beta_0}$ is a unit of $\mathbb{C}\{x,y\}$. It implies that for every $(\alpha,\beta)$ in $ W_{I}^{\pi}$ there exists $h_{\alpha \beta}$ in $\mathbb{C}\{x,y\}$ such that

 $$\pi^{*}(dF_{\alpha} \wedge dF_{\beta})=h_{\alpha \beta} \pi^{*}(dF_{\alpha_0} \wedge dF_{\beta_0}).$$ For any $(\alpha,\beta)$  in $\mathbb{P}(\mathbb{C}^k) \times \mathbb{P}(\mathbb{C}^k)$ denote by $\omega_{\alpha\beta}$ the holomorphic $2$-form $\pi^{*}(dF_{\alpha} \wedge dF_{\beta})$. A direct computation  shows that for any elements $t_1,t_2$  of $\mathbb{C}$ we have the equality $$\omega_{\alpha_0+t_1 \alpha,\beta_0+t_2 \beta}=\omega_{\alpha_{0},\beta_{0}}(1+t_1h_{\alpha,\beta_0}+t_2h_{\alpha_0,\beta}+t_1t_2h_{\alpha,\beta}). $$ This last equality implies that for a generic choice of $t_1, t_2$ the  strict transform of the polar curve $\Pi_{\alpha_0+t_1 \alpha,\beta_0+t_2 \beta}^*$ does not pass through the point $p$. It follows that $p$ is not a basepoint for the family $\Pi_F$. 
 
 Assume  that $\Pi_{\alpha \beta}^*$ passes through $p$ for every $(\alpha, \beta)$ in $W_{I}^{\pi}$ or equivalently that $S_{\alpha \beta}(0,0)=0, \forall (\alpha,\beta) \in W_{I}^{\pi}$.  Denote by $I_{p}$ the  ideal of $\mathbb{C}\{x,y\}$  generated by the finite set $\{h_{ij}(x,y):=x^{m_v(i,j)-\nu_v(I)}\phi_{ij}(x,y)\}_{i \neq j}$. Since $S_{\alpha \beta}(0,0)=0, \forall (\alpha,\beta) \in W_{I}^{\pi}$ it follows that $h_{ij}(0,0)=0$ for all $i,j\in \{1,\dots,k \}$ thus $I_p$ contains a power of the maximal ideal of $\mathbb{C}\{x,y\}$. 
     Consider the blow-up  $\mathrm{BL}_{I_p}$ of the ideal $I_{p}$. By Proposition \ref{basepointofidealvssheaf} the modification  $\mathrm{BL}_{I_p}$ has no basepoint for the family of  curves $\{\sum_{i \neq j} \gamma_{ij}h_{ij}=0\}_{[\gamma_{ij}] \in \mathbb{P}(\mathbb{C}^{k^2-k})}$. Let $O_p$ be the Zariski open set of  $\mathbb{P}(\mathbb{C}^{k^2-k})$ such that the strict transforms of the curves $\{\sum_{i \neq j} \gamma_{ij} h_{ij}=0\}_{[\gamma_{ij}] \in O_p}$ by $\mathrm{BL}_{I_p}$ do not intersect. Consider the continuous map $P:\mathbb{P}(\mathbb{C}^k) \times \mathbb{P}(\mathbb{C}^k) \longrightarrow \mathbb{P}(\mathbb{C}^{k^2-k}), P(\alpha,\beta)=[\alpha_i\beta_j]_{i \neq j}$ and consider the Zariski open set $V_p = P^{-1}(O_p)$. It follows that the strict transforms of the  curves $\{S_{\alpha \beta}=\sum_{i \neq j} \alpha_i \beta_j h_{ij}=0\}_{(\alpha,\beta) \in V_p}$ do not intersect which leads us to conclude that $\mathrm{BL}_{I_p}$ does not have basepoints for the family $\{\sum_{i \neq j} \alpha_i \beta_j h_{i j}=0\}_{(\alpha,\beta) \in \mathbb{P}(\mathbb{C}^k) \times \mathbb{P}(\mathbb{C}^k)}$. We perform the same operation for  every smooth basepoint of the family $\Pi_{F}$ in order to get the desired good resolution. A similar argument works if $p$ is a double point of $E$. \\

Now, let $G=(g_1, \dots,g_l)$ be another precomplete system of generators for $I$. Let $\pi_F:(X_F,E_F) \longrightarrow (X,0)$ be any good resolution which has no basepoints for the family $\Pi_F$. We will prove that $\pi_F$ has no basepoints for the family $\Pi_G$. Since $G$ is precomplete then for every couple $(\alpha,\beta)  \in \mathbb{P}(\mathbb{C}^k) \times \mathbb{P}(\mathbb{C}^k) $ we have \begin{equation}\label{3.1}
      \pi^*(\mathrm{d}F_{\alpha} \wedge \mathrm{d}F_{\beta})=\sum_{i \neq j} \psi_{ij}^{\alpha \beta} \pi^*(\mathrm{d}g_i \wedge \mathrm{d}g_j). \end{equation}

       Let $p$ be a smooth point of an irreducible component $E_v$ of $E_F$. We are going to prove that $p$ is not a basepoint for the family $\Pi_G$. Let $(x,y)$ be a local coordinate system of $X_F$ centered at $p$ such that $x=0$ is the equation of $E_F$. It follows from  Lemma  \ref{nu} and our hypothesis that $ \pi_F$ has no basepoints for the family $\Pi_F$, that is
     $$\pi^*(\mathrm{d}F_{\alpha} \wedge \mathrm{d}F_{\beta})=x^{\nu_v(I)} S_{\alpha \beta}(x,y) \mathrm{d}x \wedge \mathrm{d}y, \ \text{with} \ S_{\alpha \beta}(0,0) \neq 0,$$ for all $(\alpha,\beta)$ in a Zariski open set  $O_p$ of $\mathbb{P}(\mathbb{C}^k) \times \mathbb{P}(\mathbb{C}^k)$.  Then, using Equality  (\ref{3.1}), we obtain the existence of $i_0,j_0 \in \{1,\dots,l \}$ such that 
      $$ \pi^*(\mathrm{d}g_{i_0} \wedge \mathrm{d}g_{j_0})=x^{\nu_v(I)}S_{0}(x,y) \mathrm{d}x \wedge \mathrm{d}y $$ and $S_0(0,0)\neq 0$. With a similar argument as the one we  previously used in the proof of the existence of $\pi_F$ we get that $p$ is not a basepoint for the family $\Pi_G$. The same argument works if $p$ is a double point of $E_F$.

	\end{proof}

		Now, we are ready to state and prove the result we announced in Subsection \ref{section4.1}.	
	\begin{prop}\label{inner-rate-ideal}
	Let $(X,0)$ be a complex surface germ with an isolated singularity  and let $I$ be a $\mathfrak{m}$-primary ideal of $\mathcal{O}_{X,0}$. Let $F=(f_1,f_2,\dots,f_k)$ be a precomplete system  of generators for $I$. Let $ \pi: (X_{\pi},E) \longrightarrow (X,0)$ be a good resolution which factors through the blow-up  of the ideal $I$ and the principalization of the  $\mathcal{O}_{X,0}$-module $\Omega_I^2$. There exists a Zariski open set $V_{I}^{\pi}$ of $\mathbb{P}(\mathbb{C}^k) \times \mathbb{P}(\mathbb{C}^k)$ such that for any irreducible component $E_v$ of $E$  and any element $(\alpha,\beta)$ of $V_{I}^{\pi}$ we have the equality
	 $$q_{F_{\alpha},v}^{F_{\beta}}=q_v^{I}, $$ where $q_v^{I}$ is the rational number of Definition \ref{inner-rate-idealdefi}.

	\end{prop}
	
\begin{proof}[Proof of Proposition \ref{inner-rate-ideal}] Let $V_I^{\pi}=(U_I^{\pi}\times U_I^{\pi}) \cap W_{I}^{\pi}  \subset \mathbb{P}(\mathbb{C}^k)\times \mathbb{P}(\mathbb{C}^k)  $, where $U_I^{\pi},W_{I}^{\pi}$  are Zariski open sets as in Lemmas \ref{mI} and \ref{nu}. Let $E_v$ be an irreducible component of $E$ and $p$ be a smooth point of $E$ in $E_v$. By Lemma \ref{basepointofidealvssheaf} and Definition \ref{principalization} there exists $(\alpha,\beta)$ in $V_{I}^{\pi}$ such that $p \notin F_{\alpha}^* \cup F_{\beta}^* \cup \Pi_{{\alpha \beta}}^*$. By Proposition \ref{inner-rate} we have   $$ q_{F_{\alpha},v}^{F_{\beta}}=\frac{\mathrm{Ord}_{E_v}(\pi^{*} (\mathrm{d} F_{\alpha} \wedge \mathrm{d}F_{\beta}))-m_v(F_{\beta})+1}{m_v(F_{\beta})}.$$ Since $(\alpha,\beta)$ is an element of $V_{I}^{\pi}$ then by Lemmas \ref{mI} and \ref{nu} we have the equalities

$$ m_v(F_{\alpha}) = m_v(F_{\beta})=m_v(I) \ \ \  \text{and} \ \ \  \mathrm{Ord}_{E_v}(\pi^{*} (\mathrm{d} F_{\alpha} \wedge \mathrm{d}F_{\beta}))=\nu_v(I).$$ It implies that the number  $q_{F_{\alpha},v}^{F_{\beta}}$ does not depend on the choice of $(\alpha,\beta)$ in $V_{I}^{\pi}$ and
$$q_{F_{\alpha},v}^{F_{\beta}}=q_v^{I}:=\frac{\nu_v(I)-m_v(I)+1}{m_v(I)}. $$\end{proof}

To end this section, let us compare the inner rates associated with the maximal ideal $\mathfrak{m}$ of $\mathcal{O}_{X,0}$ with the inner rates of complex surface germ defined in {\cite{BFP}}. Let $(X,0) \subset (\mathbb{C}^n,0)$ be a complex surface germ with an isolated singularity at the origin of $\mathbb{C}^n$. Denote  by $\mathbb{G}(2,\mathbb{C}^n)$  the grassmanian  of complex planes of $\mathbb{C}^n$.  Consider the \textbf{Gauss map}$$\begin{array}{rcl}
\gamma: X \backslash \{0\} &\to& \mathbb{G}(2,\mathbb{C}^n), \\
p &\mapsto & T_{p}X
\end{array} $$ which associates to each point $p$ of $X \setminus \{0\}$ the tangent plane of $X$ at $p$, and take the closure of its graph: $$\mathcal{N}(X):= \overline{   \mathrm{Graph}(\gamma) } = \overline{ \{ (p,H) \in X \backslash 0 \times \mathbb{G}(2,\mathbb{C}^n)    \ | \  T_pX=H \}   }   \subset \mathbb{C}^n \times \mathbb{G}(2,\mathbb{C}^n) $$
\begin{defi}\label{classicnashmodif}
	The \textbf{Nash transform} of $X$  is the modification $\nu$ defined by
$$\begin{array}{rcl}
\nu :  \mathcal{N}(X)  &\to& X\\
(p,H)&\mapsto & p
\end{array}. $$ 
\end{defi}

\begin{prop}[{\cite[Lemma 3.2]{BFP}}]\label{bfp3.2} Let $(X,0) \subset (\mathbb{C}^n,0)$ be a complex surface germ with an isolated singularity and $ \pi: (X_{\pi},E) \longrightarrow (X,0)$ be a good resolution which factors through the Nash transform of $X$ and the blow-up of the maximal ideal of $\mathcal{O}_{X,0}$.

 There exists a rational number $q_v \in \mathbb{Q}_{>0}$ such that  for any pair of curvettes   $\gamma_1^*$ and $\gamma_2^*$  of an irreducible component $E_v$ of the exceptional divisor $E$  we have
$$ \mathrm{d}(\gamma_1 \cap \mathbb{S}_\epsilon,\gamma_2 \cap \mathbb{S}_\epsilon) = \Theta(\epsilon^{q_v}),$$where $\gamma_1=\pi(\gamma_1^*)$, $\gamma_2=\pi(\gamma_2^*)$ and $\mathrm{d}$ is the arc length distance induced by $\mathbb{C}^n$ on $X$.

Furthermore, there exists a Zariski open set $\Omega_{\pi}$ of $\mathbb{G}(n-2,\mathbb{C}^n)$ such that for any linear projection $\ell=(\ell_1,\ell_2):(X,0)\longrightarrow (\mathbb{C}^2,0)$ whose kernel is an element of $\Omega_{\pi}$ we have   $q_v=q_{v,\ell_1}^{\ell_2}$.
\end{prop}
\begin{defi}[{\cite[Definition 3.3]{BFP}}]\label{genericBFP}
The rational number $q_v$ is the \textbf{inner rate of $(X,0)$ along $E_v$}. We will say that a linear projection $\ell=(\ell_1,\ell_2)$ as in Proposition \ref{bfp3.2} is \textbf{generic with respect to $\pi$}. The polar curve of such a projection is also said to be \textbf{generic with respect to $\pi$}. In this case we also say the linear forms $\ell_1$ and $\ell_2$ are \textbf{generic} and the curves $\ell_1^{-1}(0)$ and  $\ell_2^{-1}(0)$ are said to be \textbf{generic hyperplane sections}. \end{defi}
In the following remark we will explain that the inner rate associated to the maximal ideal coincides with the inner rate in the sense of Definition \ref{genericBFP}.
\begin{rema}\label{maximalidealrate}
	The inner rate $q_v$ of $(X,0) \subset (\mathbb{C}^n,0)$ along $E_v$ coincide with the number $q_v^{\mathfrak{m}}$, where $\mathfrak{m}$ is the maximal ideal of $\mathcal{O}_{X,0}$. Indeed, consider the linear forms $${\left\{\begin{array}{rcl}
P_i :  \mathbb{C}^n  &\to& \mathbb{C}\\
(z_1,\dots,z_n) &\mapsto & z_i
\end{array} \right\}. }_{1 \leq i \leq n} $$ The family of linear forms $P=(P_1,\dots,P_n)$ is a precomplete system of generators for $\mathfrak{m}$. By Proposition \ref{inner-rate-ideal}, there exists a Zariski open set $V_{\mathfrak{m}}^{\pi}$ of $\mathbb{P}(\mathbb{C}^n) \times \mathbb{P}(\mathbb{C}^n) $ such that for any elements $(\alpha,\beta)\in V_{\mathfrak{m}}^{\pi}$ we have
 $q_v^{\mathfrak{m}}= q_{v,P_\alpha}^{P_\beta}$. Denote by $H_{\alpha,\beta}$ the kernel of the linear projection $\ell_{\alpha,\beta}=(P_{\alpha},P_{\beta})$. The set $\overline{V_{\mathfrak{m}}^{\pi}}=\{ H_{\alpha \beta} \ | \ \alpha,\beta \in V_{\mathfrak{m}}^{\pi}\}$ is a Zariski open set of $\mathbb{G}(n-2,\mathbb{C}^n)$. By Lemma \ref{bfp3.2} there exists a Zariski open set $\Omega_{\pi}$ of $\mathbb{G}(n-2,\mathbb{C}^n)$ such that $q_v=q_{v,\ell_1}^{\ell_2}$ whenever $\ell=(\ell_1,\ell_2)$ is a linear projection whose kernel is an element of $\Omega_{\pi}$. It follows that for any linear projection $\ell=(\ell_1,\ell_2)$ whose kernel is an element $\Omega_{\pi} \cap \overline{V_{\mathfrak{m}}^{\pi}}$ we have $q_v=q_{v,\ell_1}^{\ell_2}=q_{v}^{\mathfrak{m}}.$ \end{rema}

   \section{Inner rate function of an m-primary ideal}\label{sect4}
Let us  define a metric on the quasi-monomial valuations of $\mathrm{NL}(X,0)$ compatible with $I$ following \cite[Subsection 2.3]{BFP}. We start by doing it on any graph $\Gamma_{\pi}$ where $\pi:(X_{\pi},E) \longrightarrow (X,0)$ is a good resolution which factors through the blow up of $I$ and the principalization of $\Omega^2_I$.  Let us endow the dual graph  $\Gamma_{\pi}$ with a metric by declaring the length of an edge $e_{v,v'}$ to be\begin{equation*}\mathrm{length}_I(e_{v,v'})=\frac{1}{m_v(I)m_{v'}(I)}.\end{equation*}


Now, let $p$ be an intersection  point between two irreducible components  $E_v$ and $E_{v'}$ of $E$ and let $\tilde{\pi}:(X_{\tilde{\pi}},\tilde{E}) \longrightarrow (X,0)$ be the good resolution obtained by blowing up the point $p$. Then, the exceptional component $E_w$ that arises has multiplicity $m_w(I)=m_v(I)+m_{v'}(I)$ . Since $\frac{1}{m_w(I)m_v(I)}+\frac{1}{m_w(I)m_{v'}(I)}=\frac{1}{m_v(I)m_{v'}(I)}$ (by Lemma \ref{recurenceI}),  the inclusion of $\Gamma_{\pi}$ in $\Gamma_{\tilde{\pi}}$ is an isometry. Therefore, by passing to the direct limit, the metric on the dual graphs $\Gamma_{\pi}$  defines a distance denoted $\mathrm{d}_{I}$ on the quasi-monomial valuations of $\mathrm{NL}(X,0)$.
\begin{rema}
	We cannot define a distance on $\mathrm{NL}(X, 0)$, since it is not a metrizable space $($see right after Remark 2.3 of \cite{GignacRuggiero2017}$)$. However, we can define a distance on the set of quasi-onomial valuations as outlined below. This distance, however, will not be compatible with the topology of the non-archimedean link in the sense that it will define a stronger topology. See \cite[Theorem 2.29]{GignacRuggiero2017}.
\end{rema}

\begin{defi}\label{skeletalmetricofI}
	We call $d_I$ the \textbf{skeletal metric on $\mathrm{NL}(X,0)$ with respect to $I$}. When we want to specify the distance on the quasi-monomial valuations of  $\mathrm{NL}(X,0)$ we will adopt the notation $(\mathrm{NL}(X,0),\mathrm{d}_{I})$.
\end{defi}
\begin{prop}\label{inner-rate functionI} There exists a unique continuous function  $$ \mathcal{I}_{I} : (\mathrm{NL}(X,0),\mathrm{d}_I) \longrightarrow \mathbb{R}_{>0} \cup \{ +\infty  \} $$ such that $\mathcal{I}_I(v)=q_{v}^{I}$  for every divisorial point $v$ of $\mathrm{NL}(X,0)$. If $\pi$ is a good resolution of $(X,0)$ which factors through the blow-up of  $I$ and the principalization  of $\Omega_{I}^2$ then $\mathcal{I}_{I}$ is linear on the edges of $\Gamma_{\pi}$.\end{prop}
\begin{defi}
	The function $\mathcal{I}_I$ is called \textbf{the inner rates function with respect to the ideal} $I$.
\end{defi}

We get Proposition \ref{inner-rate functionI} by going through the exact same proof as  \cite[Lemma 3.8]{BFP},  using the following Lemma \ref{recurenceI} instead of \cite[Lemma 3.6]{BFP}.
\begin{lemm}\label{recurenceI}
	
	Let $I$ be an $\mathfrak{m}$-primary ideal for $\mathcal{O}_{X,0}$ and $\pi :(X_{\pi},E) \longrightarrow (X,0)$ be a good resolution of $(X,0)$ which factors through the blow-up of $I$ and the principalization of $\Omega_{I}^2$. Let $E_v$ be an irreducible component of $E$. Let $p$ be a point of  $E_v$ and let $E_{w}$ be the exceptional component created by the blow-up of $X_{\pi}$ at $p$. Then:
	\begin{enumerate}
		\item if $p$ is a smooth point of $E$, then
		\begin{align*}
			m_w(I)=m_v(I)  &&  \text{and} && q_{w}^{I}=q_{v}^I+\frac{1}{m_v(I)}.
		\end{align*}\label{recurenceIpoint1}

		\item if $p$ is a double point of $E_v$ and $E_{v'}$, then
		\begin{align*}
			m_w(I)=m_v(I)+ m_{v'}(I) &&\text{and}&& q_{w}^I=\frac{q_{v}^I m_v(I) + q_{v'}^I m_{v'}(I)}{m_v(I)+m_{v'}(I)}.
		\end{align*}\label{recurenceIpoint2}				
	\end{enumerate}
\end{lemm}

\begin{proof}
	Let \( F = (f_1, \dots, f_k) \) be a precomplete system of generators for \( I \). Let \( V_{I,p}^{\pi} \) be the subset of the open Zariski set \( V_{I}^{\pi} \) of Proposition \(\ref{inner-rate-ideal}\) such that for all elements \((\alpha, \beta) \) in \( V_{I,p}^{\pi} \), we have
	\[ 
	p \notin \left( F_{\alpha}^* \cup F_{\beta}^* \cup \Pi_{\alpha \beta}^* \right). 
	\]
	Since \( \pi \) factors through the blow-up of \( I \) and the principalization of \( \Omega^2_I \) by Lemma \(\ref{basepointofidealvssheaf}\) and Definition \(\ref{principalization}\), \( V_{I,p}^{\pi} \) is also a Zariski open set of \( \mathbb{P}(\mathbb{C}^k) \times \mathbb{P}(\mathbb{C}^k) \).
	
	\underline{Proof of \(\ref{recurenceIpoint1}\):}
	
	Assume that \( p \) is a smooth point of an irreducible component \( E_v \) of \( E \). Let \((x,y)\) be a local system of coordinates centered at \( p \) such that \( x=0 \) is the local equation of \( E_v \). Let \((\alpha,\beta)\) be an element of \( V_{I,p}^{\pi} \). In this case,
	$$ 
	F_{\alpha} \circ \pi(x,y) = x^{m_v(I)} U_{\alpha}(x,y), \quad F_{\beta} \circ \pi(x,y) = x^{m_v(I)} U_{\beta}(x,y) 
	$$
	where $ U_{\alpha}, U_{\beta} $ are units of $ \mathcal{O}_{X_{\pi},p} $. Let $ \mathrm{BL}_p $ be the blow-up centered at $p$ and denote by $ E_w $ the resulting irreducible component. In the coordinates chart $(x,y) = (x',x'y')$, we have
	
	\begin{eqnarray*}	
	F_{\alpha} \circ \pi \circ \mathrm{BL}_p(x',y') = x'^{m_v(I)} U_{\alpha}(x',x'y'), \\   F_{\beta} \circ \pi \circ \mathrm{BL}_p(x',y') = x'^{m_v(I)} U_{\beta}(x',x'y')
	\end{eqnarray*}
	which means that \( m_w(I) = m_v(I) \). Furthermore, since \( F_{\alpha}^*, F_{\beta}^* \) and \( \Pi_{\alpha,\beta}^* \) do not pass through the point \( p \), it follows that they do not meet the component \( E_w \). We then get the equality  
	\[ 
	q_{w}^{I} = q_{v}^I + \frac{1}{m_v(I)} 
	\]
	as a direct consequence of Theorem \(\ref{laplacien}\) applied to the component $E_w$.
	
	\underline{Proof of \(\ref{recurenceIpoint2}\):}
	
	Assume that \( p \) is an intersecting point of two irreducible components \( E_{v} \) and \( E_{v'} \) of \( E \). Let \((x,y)\) be a local system of coordinates centered at \( p \) such that \( x=0 \) and \( y=0 \) are the local equations of \( E_{v} \) and \( E_{v'} \), respectively. Let \((\alpha,\beta)\) be an element of \( V_{I,p}^{\pi} \). In this case,
	\[ 
	F_{\alpha} \circ \pi(x,y) = x^{m_{v}(I)} y^{m_{v'}(I)} U_{\alpha}(x,y), \quad F_{\beta} \circ \pi(x,y) = x^{m_{v}(I)} y^{m_{v'}(I)} U_{\beta}(x,y) 
	\]
	where \( U_{\alpha}, U_{\beta} \) are units of $ \mathcal{O}_{X_{\pi},p} $. Let \( \mathrm{BL}_p \) be the blow-up centered at \( p \) and denote by \( E_w \) the resulting irreducible component. In the coordinates chart \((x,y) = (x',x'y')\), we have
	$$ 
	F_{\alpha} \circ \pi \circ \mathrm{BL}_p(x',y') = x'^{m_{v}(I)} (x'y')^{m_{v'}(I)} U_{\alpha}(x',x'y')$$  $$F_{\beta} \circ \pi \circ \mathrm{BL}_p(x',y') = x'^{m_v(I)} (x'y')^{m_{v'}(I)} U_{\beta}(x',x'y')
	$$
	which means that \( m_w(I) = m_{v}(I) + m_{v'}(I) \). Furthermore, since \( F_{\alpha}^*, F_{\beta}^* \) and \( \Pi_{\alpha,\beta}^* \) do not pass through the point \( p \), it follows that they do not meet the component \( E_w \). We then get the equality  
	\[ 
	q_{w}^I = \frac{q_{v}^I m_v(I) + q_{v'}^I m_{v'}(I)}{m_v(I) + m_{v'}(I)}
	\]
	as a direct consequence of Theorem \(\ref{laplacien}\) applied again to the component $E_w$.\end{proof}

\begin{proof}[Proof of Proposition \ref{inner-rate functionI}]

	Let $I$ be an $\mathfrak{m}$ primary ideal of $\mathcal{O}_{X,0}$ and $\pi:(X_{\pi},E) \longrightarrow (X,0) $ be a good resolution of $(X,0)$ which factors through the blow-up of $I$ and the principalization of $\Omega_I^2$.
	We only need to show that the inner rates with respect to $I$ extend uniquely to a continuous map on $\Gamma_{\pi}$ which is linear on its edges with integral slopes, because $\mathrm{NL}(X,0)$ is homeomorphic to the inverse limit of dual graphs.\\
	Let $e_{v v'}$ be an edge of $\Gamma_{\pi} $. The  slope of $\mathcal{I}_I$ on $e_{v,v'}$ is $ \frac{q_{v' }^I - q_{v }^I  }{\mathrm{length}_{I}(e_{v v' } )}= m_v(I)m_{v'}(I)(q_{v' }^I - q_{v }^I )$ and it is an integer by Definition  \ref{inner-rate-idealdefi}.
	
	By density of the divisorial valuations (Remark \ref{densiteremarque}), it is sufficient to prove the linearity of $\mathcal{I}_I$ along  $e_{v v'}$ on the set of the divisorial valuations. On the other hand, since every divisorial point of that edge are obtained by successively  blowing up  its double points, it is sufficient to prove that $\mathcal{I}_{I}$  is linear on the set $\{ v, v',v''\}$, where $v''$ is obtained by blowing up the double point $E_{v} \cap E_{v'}$. Therefore, all we have to show is that 
	\begin{equation}\label{funk} \frac{\mathcal{I}_I(v') - \mathcal{I}_I(v)}{ \mathrm{length}_I(e_{v v'} ) } =\frac{\mathcal{I}_I(v'') - \mathcal{I}_I(v)}{ \mathrm{length}_I(e_{v v''} ) }.\end{equation} Since the good resolution $\pi$ factors through the blow-up of $I$ and the principalization of $\Omega_I^2$  the equality \eqref{funk}  is a direct consequence of point \ref{recurenceIpoint2} of Lemma \ref{recurenceI}. \end{proof}

Now, let us state the principal result of this section, which  tells us that the inner rate function associated to an $\mathfrak{m}$-primary ideal is completely determined by its values on the vertices of the dual graph of the minimal good resolution of $(X,0)$ which factors through the blow-up of $I$ and the principalization of the  $\mathcal{O}_{X,0}$-module $\Omega_{I}^2$.
\begin{prop}\label{nashvsrates}
	Let $I$ and $J$ be two $\mathfrak{m}$-primary ideals of the local ring $\mathcal{O}_{X,0}$. Let $\pi:(X_{\pi},E) \longrightarrow (X,0)$ be the minimal good resolution which factors through the blow-up of $I$ and the principalization of $\Omega^2_I$.  
	
	Then, $\mathcal{I}_{I}=\mathcal{I}_{J}$ if and only if the following conditions are verified
	
	\begin{itemize}
		\item The good resolution $\pi$ is the minimal good resolution which factors through the blow-up of $J$ and the principalization of  $\Omega^2_J$.
		\item $m_v(I)=m_v(J)$ and   $q_v^{I}=q_v^{J}$  for all vertices $v$ of $V(\Gamma_{\pi})$.	
	\end{itemize} 		
\end{prop}

The following preliminary lemma is required to prove Proposition \ref{nashvsrates} and will have another interest in the last section of the paper.

\begin{lemm}\label{integralclosure}
	
	Let $I$ and $J$	be two $\mathfrak{m}$-primary ideals of $\mathcal{O}_{X,0}$. Let $\overline{\mathrm{BL}_I}:(\overline{X_I},\overline{E_I}) \longrightarrow (X,0)$ and  $\overline{\mathrm{BL}_J}:(\overline{X_J},\overline{E_J}) \longrightarrow (X,0)$ be the normalized blow-ups of $I$ and $J$ respectively. If $\mathcal{I}_{I}=\mathcal{I}_J$ then $I.\mathcal{O}_{\overline{X_I}}=	J.\mathcal{O}_{\overline{X_I}}$, $I.\mathcal{O}_{\overline{X_J}}=	J.\mathcal{O}_{\overline{X_J}}$ and the modifications $\overline{\mathrm{BL}_{I}}$ and $\overline{\mathrm{BL}_J}$ are isomorphic i.e, each one factors through the other.\end{lemm}
\begin{proof}
	Let $\pi:(X_{\pi},E)\longrightarrow (X,0)$ be a good resolution which factors through the blow-ups of the ideals $I$ and $J$ and the principalizations of the  $\mathcal{O}_{X,0}$-modules $\Omega_I^2$ and $\Omega_J^2$. We first prove that the ideal sheaves $I.\mathcal{O}_{X_{\pi}}=\mathcal{O}_{X_{\pi}}(-\sum_{V}\mathrm{Ord}_{V}(I) V)$ and $J.\mathcal{O}_{X_{\pi}}=\mathcal{O}(-\sum_{V}\mathrm{Ord}_{V}(J) V)$ are equal, where $\mathrm{Ord}_{V}(I) := \mathrm{min} \{\mathrm{Ord}_V(h) \ | \ h \in I\}$ whenever $V$ is a curve embedded in $X_{\pi}$ . Since $\pi$ factors through the blow-ups of $I$ and $J$ then, by Proposition \ref{hironakablowupideal}, the ideal sheaves  $I.\mathcal{O}_{X_{\pi}}$ and $J.\mathcal{O}_{X_{\pi}}$  are locally principal.

	It remains to prove that their associated divisors are equal. Let $E_v$ be an irreducible component of $E$ and let us prove that $m_v(I)=m_v(J)$.  Consider the blow-up centered at a smooth point of the component $E_v$ and denote by $E_w$ the resulting component. By hypothesis we have $q_{w}^{I}=q_{w}^J$ and it follow by Lemma \ref{recurenceI} that  $$ q_w^{I}=q_v^{I}+\frac{1}{m_v(I)}=q_v^{J}+\frac{1}{m_v(J)}=q_w^{J}.$$ The equality below implies that $m_v(I)=m_v(J)$ because by hypothesis $q_v^I=q_v^J$. It follows that $$I.\mathcal{O}_{X_{\pi}}=	J.\mathcal{O}_{X_{\pi}}.$$

	Denote by $\overline{\mathrm{BL}_I}:(\overline{X_I},\overline{E_I}) \longrightarrow (X,0)$ the normalized blow-up of $I$. We will prove that $I. \mathcal{O}_{\overline{X_I}}=J. \mathcal{O}_{\overline{X_I}}$. Let $r:(X_{\pi},E)\longrightarrow (\overline{X_I},\overline{E}_I)$ be the modification such that $ \overline{\mathrm{BL}}_I \circ r = \pi$. Let $p$ be a point of  $\overline{E}_{I}$. Then, there exists $h \in (\mathcal{O}_{\overline{X_I}})_p$ such that $ (I. \mathcal{O}_{\overline{X_I}})_p$ is generated by $h$. Then $h$ generates $ (I. \mathcal{O}_{\overline{X_I}})_q$ for all $q$ in a neighborhood of $p$. Since $r$ is an isomorphism  when restricting the image to a punctured neighborhood $U_p$ of $p$ in $\overline{X_I}$, $h$ must divide all functions of $ (J. \mathcal{O}_{\overline{X_I}})_p$, that is, for all $\phi \in (J. \mathcal{O}_{\overline{X_I}})_p$, we have that $\phi = h . g$, for some function $g$. We claim that there exists a function $\phi$ such that $g$ is a unit. In fact, if this was not true, then after applying  the modification $r$   we would have that $I.\mathcal{O}_{X_{\pi}}$ and $J.\mathcal{O}_{X_{\pi}}$ are distinct and it is a contradiction.	 We have proved that $I.\mathcal{O}_{\overline{X_I}}=	J.\mathcal{O}_{\overline{X_I}}$ and by symmetry of the roles of $I$ and $J$ we have $I.\mathcal{O}_{\overline{X_J}}=	J.\mathcal{O}_{\overline{X_J}}$. \\
	
	By Proposition \ref{hironakablowupideal} the sheaf $I.\mathcal{O}_{\overline{X_I}}$ is locally principal which means that it is also the case for $J.\mathcal{O}_{\overline{X_I}}$. It implies that $\overline{\mathrm{BL}_I}$ factors through $\overline{\mathrm{BL}_J}$. We can also prove symmetrically that $\overline{\mathrm{BL}_J}$ factors through $\overline{\mathrm{BL}_I}$.\end{proof}

Now, we are ready to prove the main result of this section.
\begin{proof}[Proof of Proposition \ref{nashvsrates}]
	$\bullet$ Assume that $\mathcal{I}_{I}=\mathcal{I}_{J}$ and let us prove first that $\pi$ factors through the blow-up of $J$ and the principalization of $\Omega_J^2$. By Lemma \ref{integralclosure} the modifications $\overline{\mathrm{BL}_I}$ and $\overline{\mathrm{BL}_J}$ are isomorphic i.e, factors through each other. It follows that $\pi$ factors through $\mathrm{BL}_J$.

	It remains to prove that $\pi$ factors through the principalization of $\Omega_{J}^2$. Let us suppose that there exists $F=(f_1, \dots ,f_k)$ a precomplete system of generators for $J$ such that $\pi$ has a basepoint for the family of polar curves $\Pi_{F}$. Let $p$ be a basepoint of $\pi$ for the family $\Pi_{F}$.

	Assume first that $p$ is a smooth point of $E$. Denote by $E_{w}$ the irreducible component that arises from the blow-up of center $p$. By Lemma \ref{recurenceI} we get that \begin{eqnarray}\label{4.2}  q_w^{I}=q_v^{I}+\frac{1}{m_v(I)}\end{eqnarray}

	\noindent and by Theorem \ref{laplacien} applied to $E_w$ we get
	
	\begin{eqnarray}\label{4.3} 
		q_w^{J}=q_v^{J} +\frac{1+\Pi_{\alpha \beta}^* \cdot E_w}{m_v(J)}.
	\end{eqnarray}

	On the other hand, we have $I.\mathcal{O}_{X_{\pi}} = J.\mathcal{O}_{X_{\pi}}$, indeed, by Lemma \ref{integralclosure} the ideal sheaves $I.\mathcal{O}_{\overline{X_{I}}}$ and $J.\mathcal{O}_{\overline{X_{I}}}$ are equal. It follows that their pull back in $X_{\pi}$ are also equal. In particular,  $m_v(I)=m_v(J)$ and we have by hypothesis $q_v^{I}=	 q_v^{J}$ and $q_w^{I}=	 q_w^{J}$. This contradicts the equalities (\ref{4.2}) and (\ref{4.3}) and it  proves that $p$ cannot be a basepoint for the family $ \Pi_{F} $.

	Assume now that $p$ is a double point of $E$. Let  $E_{v_1}$	and $E_{v_2}$ be the irreducible components of $E$ which intersects at $p$. Let $E_w$ be the irreducible component  that arrises from the blow-up of the point $p$. Lemma \ref{recurenceI} implies that 	\begin{eqnarray}\label{4.4} 			q_w^{I}=\frac{q_{v_{1}}^{I}+q_{v_{2}}^{I}}{m_{v_{1}}(I)+m_{v_2}(I)}.
	\end{eqnarray}
	On the other hand the inner rate formula \ref{laplacien} gives
	
	\begin{eqnarray}\label{4.5} 			q_w^{J}=\frac{\Pi_{\alpha \beta}^* \cdot E_w + q_{v_{1}}^{J}+q_{v_{2}}^{J}}{m_{v_{1}}(J)+m_{v_2}(J)}.	\end{eqnarray}
	Again, the equalities (\ref{4.4}) and (\ref{4.5}) contradicts the hypothesis  $\mathcal{I}_I=\mathcal{I}_J$. We  proved that $\pi$ factors through the principalization of $\Omega^2_J$.
	
	We have the equalities $m_v(I)=m_v(J)$ and $q_v^{I}=q_v^J$ for all $v\in V(\Gamma_\pi)$ because $\mathcal{I}_I = \mathcal{I}_J$.

	$\bullet$	Conversely, if $\pi$ factors through the blow-ups of  $I$ and $J$ and the principalizations of $ \Omega_2^I$ and  $\Omega_2^{J}$ and if $q_v^{I}=q_v^J, m_v(I)=m_v(J)$ for all  vertices  $v$ of $\Gamma_{\pi}$ then it follows directly from Lemma \ref{recurenceI} that $\mathcal{I}_I=\mathcal{I}_J$. \end{proof}







\section{Geometric interpretation of the inner rates associated to an ideal}

  		Let $(X,0)$ be a complex surface germ with an isolated singularity and $I$ be an $\mathfrak{m}$-primary ideal of $\mathcal{O}_{X,0}$. The aim of this section is to give a geometric interpretation of the inner rates associated to $I$ as the inner rates of a complex surface germ with an isolated singularity embedded in $\mathbb{C}^n$ (Theorem \ref{immersion}).

Now, we are ready to state  the main result of this section which enables us to introduce in Definition \ref{completeimmersion} the notion  of \textit{complete generator system for $I$}. Recall that the skeletal metric $d_I$ was defined in Definition \ref{skeletalmetricofI}.

\begin{thm}[Theorem \ref{thmz}]\label{immersion}
	Let $(X,0)$ be a complex surface germ with an isolated singularity and $I$ be a  $\mathfrak{m}$-primary ideal of $\mathcal{O}_{X,0}$. There exists a precomplete system of generators $F=(f_1,f_2,\dots,f_k)$ of $I$ such that the holomorphic map		
	$$\begin{array}{rcl}
F: (X,0) &\to& (F(X),0) \subset (\mathbb{C}^k,0) \\
p &\mapsto &(f_1(p),f_2(p),\dots,f_k(p))
\end{array} $$
verifies the following properties:
\begin{enumerate}
\item The image $(F(X),0)$ is a complex surface germ with an isolated singularity at the origin of $\mathbb{C}^k$. \label{point111}
\item The map $F$ is a homeomorphism on its image and a modification of $(F(X),0)$. \label{point222}
\item The induced map $\tilde{F}:(\mathrm{NL}(X,0), \mathrm{d}_I) \longrightarrow (\mathrm{NL}(F(X),0),\mathrm{d}_{\mathfrak{m}_F})$ is an isometry which 	makes the following  diagram commute 
	$$ \xymatrix{
   ( \mathrm{NL}(X,0), \mathrm{d}_I) \ar[r]^{\tilde{F}}  \ar[rd]^{\mathcal{I}_{I}} & (\mathrm{NL}(F(X),0), \mathrm{d}_{\mathfrak{m}_F} ) \ar[d]^{\mathcal{I}_{\mathfrak{m}_{F}}} \\
      & \mathbb{R}_{+}^* \cup \infty
  },	$$ where $\mathfrak{m}_F$ is the maximal ideal of $\mathcal{O}_{F(X),0}$. \label{333}
  \end{enumerate}
   \end{thm} \begin{rema}\label{remarkokanormalization}
If we assume in the statement of Theorem \ref{immersion}	 that the germ $(X,0)$ is normal then the morphism $F$ is the normalization of the complex surface germ $(F(X),0)$ because it is, in particular, a finite modification. \end{rema}

\begin{defi}\label{completeimmersion}
	We will say that a system of generators for $I$ is \textbf{complete} if it verifies  the properties \ref{point111}, \ref{point222} and \ref{333} of Theorem \ref{immersion}.	\end{defi}  This notion is a strengthening of the notion of \emph{precompleteness} introduced in Definition \ref{precompletegeneratorsystem}.   

Given a complex surface germ $(X,0)$ with an isolated singularity, Theorem \ref{immersion} enables us to construct complex surface germs homeomorphic to $(X,0)$ whose metric invariants (inner rates of complex surface germs) are prescribed by an ideal of $\mathcal{O}_{X,0}$. Besides providing a nice geometric interpretation for the inner rates associated with an ideal, we will later see that it is a powerful tool for explicitly constructing many complex surface germs that are homeomorphic to a given one but have distinct metric types.

In order to prove Theorem \ref{immersion} we also need the following theorem of Remmert which gives us a sufficient condition for the image of a complex space by a holomorphic map to be a complex space.
	\begin{thm}[{\cite[Remmert's proper mapping theorem]{Demailly}}]\label{Grauert}
	
			Let $f:X \longrightarrow Y$ be a proper holomorphic map. Then $f(X)$ is an analytic set embedded in $Y$.
	\end{thm}
		\begin{exam}
		For example, one can define the Whitney umbrella as the image of $\mathbb{C}^2$ by the map $(x,y) \in \mathbb{C}^2 \mapsto (xy,x,y^2)$ and whose equation is $z_1^2=z_2^2z_3$. But if we choose a non proper holomorphic map $(x,y) \in \mathbb{C}^2 \mapsto (xy,y^2)$ then the image  $\mathbb{C}^2 - \{(x,0) \ |  x \neq 0\}$ is dense thus not an analytic set embedded in $\mathbb{C}^2$.
	\end{exam}

	\begin{proof}[Proof of Theorem \ref{immersion}] Let $(f_1,\dots,f_k)$ be a precomplete system of generators for $I$.
	
	$\bullet$  We will first prove that the morphism $$\begin{array}{rcl}
G: (X ,0) &\to& (\mathbb{C}^k,0) \\
p &\mapsto &(f_1(p),f_2(p) \dots                                                                                                                                                             ,f_k(p))
\end{array} $$ is an immersion on $X \backslash 0$.

 Assume that $G$ is not an immersion. It follows that there exists a point $p$ of $X \backslash 0$ such that $\mathrm{d}f_i \wedge \mathrm{d}f_j(p)=0$ for all $i,j=1, \dots, k$.

 Let $(\ell_1,\dots,\ell_r)$ be a system of generators for the maximal ideal $\mathfrak{m}$ of $\mathcal{O}_{X,0}$. Therefore the following morphism is an embedding  $$\begin{array}{rcl}
\ell: (X ,0) &\to& (\mathbb{C}^r,0) \\
p &\mapsto &(\ell_1(p),\ell_2(p), \dots                                                                                                                                                             ,\ell_r(p))
\end{array}. $$  Since $\ell$ is in particular a biholomorphism on its image there exists  $i,j,m \in \{1, \dots ,r\}$ such that $ \ell_m(p) \neq 0$ and $\mathrm{d}\ell_i \wedge \mathrm{d}\ell_j (p) \neq 0  $.\\

By hypothesis, there exists a natural number $n$ such that $\mathfrak{m}^n \subset I$. Consider the functions $F_{im}^{(ab)}=(a\ell_i+b \ell_m)^n$ and $F_{jm}^{(cd)}=(c\ell_j+d \ell_m)^n$, where $a,b,c,d \in \mathbb{C}$. Since  $F_{im}^{(ab)}$ and $F_{jm}^{(cd)}$ are elements of $I$ then $\omega=\mathrm{d}F_{im}^{(ab)} \wedge \mathrm{d}F_{jm}^{(cd)}$ is an element of $\Omega_I^2$
 $$\omega=n^2(a\ell_i+ b \ell_m)^{n-1}( c\ell_j+ d \ell_m )^{n-1}\left(ac (\mathrm{d}\ell_i \wedge\mathrm{d} \ell_j) +ad(\mathrm{d}\ell_i \wedge\mathrm{d} \ell_m)+bc(\mathrm{d}\ell_m \wedge\mathrm{d} \ell_j)\right).$$ Let us choose $a,b,c$ and $d$ such that $\omega(p) \neq 0$. By assumption, $(f_1, \dots, f_k)$ is a precomplete system of generators for $I$ hence there exists a family of  functions $\{\psi_{ij}\}_{1 \leq i,j\leq k}$ such that $$\omega= \sum_{1 \leq i,j\leq k} \psi_{ij}\mathrm{d}f_i \wedge \mathrm{d}f_j.$$ This last equality cannot happen if $\mathrm{d}f_i \wedge \mathrm{d}f_j(p)=0$ for all $1 \leq i,j\leq k$. This contradiction shows that the map $G$ is an immersion on $X \backslash 0$.\\

But the map $G$ is not necessarily injective,  let us  then consider the following holomorphic map  
$$\begin{array}{rcl}
F: (X ,0) &\to& (\mathbb{C}^{k+kr},0) \\
p &\mapsto &(f_1(p),f_2(p) \dots                                                                                                                                                             ,f_k(p),f_1(p)\ell_1(p),f_1(p)\ell_2(p),\\
& &\dots,f_1(p)\ell_r(p), \dots,f_k(p) \ell_1(p),\dots, f_k(p) \ell_r(p) ).\end{array} 
$$

 We can directly check that $F$ is injective by using the fact that $(\ell_1(p_1), \dots, \ell_r(p_1)) \neq (\ell_1(p_2), \dots, \ell_r(p_2))$ for all $p_1 \neq p_2$ in $X$ . 
 
$\bullet$ Let us prove that $F$ is proper. For that, first we prove that $F$ is a local biholomorphism from $X \backslash 0$ to $F(X) \backslash 0$ . Let $p_0$ be a non singular point of $X$. There exists $i_0 \in \{1,\dots,k\}$ such that $f_{i_0}(p_0) \neq 0$ because the ideal $I$ is $\mathfrak{m}$-primary. Denote by $U_{0}$ an open neighborhood of $p_0$ such that $f_{i_0}$ does not vanish on $U_0$. Let us consider the holomorphic map $$h:\mathbb{C}^{k+rk} \backslash \{ z_{i_0} \neq 0\}\longrightarrow \mathbb{C}^r,$$ $$ h(z_1, \dots, z_k, z_{1,1},z_{1,2}, \dots,z_{1,r},\dots,z_{k,1},\dots,z_{k,r})=\left(\frac{z_{i_{0},1}}{z_{i_0}}, \dots , \frac{z_{i_{0},r}}{z_{i_0} }\right).$$ 
The composed map $h \circ F_{|U_0}(p)=(\ell_1(p),\dots, \ell_r(p))=\ell_{|U_0}$ is a biholomorphism from $U_0$ to $\ell(U_0)$ by definition of $\ell$. We have the equality $F^{-1}_{|F(U_0)}=(\ell^{-1} \circ h)_{| F(U_0)}$ which implies that $F_{|U_0}$ is a biholomorphism on its image. It follows that $F_{|X\backslash 0}:X\backslash 0 \longrightarrow F(X) \backslash 0$ is a biholomorphism. The map $F$ is continuous at $0$ because it is holomorphic. It implies that $F$ is a proper homeomorphism and by Theorem \ref{Grauert} we get that $(F(X),0)$ is a complex surface germ with an isolated singularity. We proved that the system of generators $F_I=(f_1,f_2 \dots                                                                                                                                                             ,f_k,f_1\ell_1,f_1\ell_2,\dots,f_1\ell_r,\dots,f_k \ell_1,\dots, f_k \ell_r)$ of the ideal $I$ is complete.

$\bullet$	Now, let us study the induced map \(\tilde{F}: \mathrm{NL}(X,0) \longrightarrow \mathrm{NL}(F(X),0)\) defined on the non-archimedean link of \((X,0)\). Since \(F\) is also a modification, the comorphism \(F^*:\mathrm{Frac}(\mathcal{O}_{F(X),F(p)}) \longrightarrow \mathrm{Frac}(\mathcal{O}_{X,p})\) is an isomorphism. Therefore, the map \(\tilde{F}\) is, by Proposition \ref{onetoonearchimede}, a one-to-one continuous map.

		Let us prove that $\mathcal{I}_{\mathfrak{m}_{F}} \circ \tilde{F}=\mathcal{I}_{I}$.  Let $val$ be a divisorial valuation of $\mathrm{NL}(X,0)$. Let $\pi:(X_{\pi},E) \longrightarrow (X,0)$ be a good resolution such that there exists an irreducible component $E_v$ of $E$ whose associated divisorial valuation is $val$. The image of $val$ by  $\tilde{F}$ is the valuation $$\tilde{F}(val)(h)= \frac{m_v(h \circ F \circ \pi)}{m_v(\mathfrak{m})} \frac{1}{val(F^*\mathfrak{m}_F)}=\frac{m_v(h \circ F \circ \pi)}{m_v(I)}, h \in \mathcal{O}_{F(X),0},$$ because $$val(F^*\mathfrak{m}_F)=val(I)=\frac{m_v(I)}{m_v(\mathfrak{m})}.$$ On the other hand, since $F$ is a modification, the map $\pi_F=F \circ \pi :(X_{\pi},E) \longrightarrow (F(X),0)$ is a good resolution of $(F(X),0)$. Denote by $val_F$ the divisorial valuation of $\mathrm{NL}(F(X),0)$ associated to the irreducible component $E_v$ of $E$. Let $h$ be an element of $\mathcal{O}_{F(X),0}$, we have \begin{eqnarray}\label{13003} val_F(h)=\frac{m_v(h \circ F \circ \pi)}{m_v(m_{F})}=\tilde{F}(val)(h), \end{eqnarray} because \begin{eqnarray}\label{13000} m_v(\mathfrak{m}_{F})=\mathrm{inf}\{\mathrm{Ord}_{E_v}(h \circ F \circ \pi) \ | \ h \in  \mathfrak{m}_F\} =m_v(I).\end{eqnarray}


		Furthermore, we have \begin{eqnarray}\label{13002} \nu_v(\mathfrak{m}_{F})=\mathrm{inf}\{\mathrm{Ord}_{E_v}(\pi_F^*(\mathrm{d}g\wedge  \mathrm{d}f)) \ | \ g,f \in  \mathfrak{m}_F\} =\nu_v(I),\end{eqnarray}				
		because  $F^{*}\mathfrak{m}_F=I$. It follows from  (\ref{13003}), (\ref{13000}) and (\ref{13002}) that $$ \mathcal{I}_{\mathfrak{m}_F}\circ \tilde{F}(val)=\mathcal{I}_{\mathfrak{m}_F}(val_F) =q_{v}^{\mathfrak{m}_F}=\frac{\nu_v(\mathfrak{m}_F)-m_v(\mathfrak{m}_F)+1}{m_v(\mathfrak{m}_F)}=\mathcal{I}_{I}(val),$$ because $\nu_v(\mathfrak{m}_F) =\nu_v(I)$ and $m_v(\mathfrak{m}_F) =m_v(I)$.

		 It remains to prove that $\tilde{F}:(\mathrm{NL}(X,0),\mathrm{d}_{I}) \longrightarrow (\mathrm{NL}(F(X),0),\mathrm{d}_{\mathfrak{m}_F})$ is an isometry. Let $v$ and $v'$ be two adjacent vertices of $\Gamma_{\pi}$. From what we done before, the image of the edge $e_{v,v'}$ by $\tilde{F}$ is again the edge $e_{v,v'}$ and we have:
		
		$$\mathrm{d}_{I}(v,v')=\frac{1}{m_v(I)m_{v'}(I)}=\frac{1}{m_v(\mathfrak{m}_F)m_{v'}(\mathfrak{m}_F)}=\mathrm{d}_{\mathfrak{m}_F}(v,v').$$ It follows that $\tilde{F}$ is an isometry on the divisorial valuations and it follows by their density in $\mathrm{NL}(X,0)$ that it is an isometry. \end{proof}

	We end this section with an example which shows that a precomplete system of generators for an ideal is not necessarily complete.
	\begin{exam}
		We have seen  in Example \ref{precompleteexample} that the set $G=(x^2,y^2,xy^2,x^2y)$ is a precomplete system of generators of the ideal   $I:=(x^2,y^2)$ of $\mathcal{O}_{\mathbb{C}^2,0}$. But $G$ is not complete. Indeed, the morphism $(x,y) \in \mathbb{C}^2 \mapsto (x^2,y^2,xy^2,x^2y)$ is not injective because the points $(1,0)$ and $(-1,0)$ are both sent to the point $(1,0,0,0)$.

		The set of generators $F=(x^2,y^2,xy^2,x^2y,x^3,y^3)$ is a complete set of generators for $I$ and the image of the morphism $(x,y) \in \mathbb{C}^2 \mapsto (x^2,y^2,xy^2,x^2y,x^3,y^3) \in \mathbb{C}^6$ is a complex surface germ $(X_I,0)$ embedded in $\mathbb{C}^6$ with an isolated singularity at the origin. \end{exam}

		\section{Application to polar exploration and outer Lipschitz geometry}
	Now we will apply Theorem \ref{immersion} to polar exploration of complex surface germs in the sense of the paper by Belotto, Fantini, N\'emethi, and Pichon in \cite{BFNP}. 

Let $(X,0)$ be a complex surface germ embedded in $(\mathbb{C}^k,0)$ and consider $\pi:(X_{\pi},E) \longrightarrow (X,0)$ a good resolution of $(X,0)$, and denote by $E_{v_1}, \dots, E_{v_n}$ the irreducible components of $E$. Let $\ell=(\ell_1,\ell_2):(X,0) \subset (\mathbb{C}^k,0) \longrightarrow (\mathbb{C}^2,0)$ be a generic projection with respect to $\pi$ in the sense of Definition \ref{genericBFP}. Denote by $\Pi$ the generic polar curve of $\ell$. As in \cite{BFNP}, we associate to $\pi$ a triplet $(\Gamma_{\pi},L_{\pi},P_{\pi})$, where $\Gamma_{\pi}$ is the dual graph of $\pi$, and vectors $L_{\pi}:=(\ell_{1}^* \cdot E_{v_i})_{1 \leq i \leq n}$ and $P_{\pi}:=(\Pi^* \cdot E_{v_i})_{1 \leq i \leq n}$ with $\ell_{1}^*$ and $\Pi^*$ being, respectively, the strict transforms of $\ell_{1}^{-1}(0)$ and the polar curve $\Pi$ by $\pi$. Belotto, Fantini, N\'emethi, and Pichon proved in \cite{BFNP} the following theorem.

	\begin{thm}[{\cite[Theorem A]{BFNP}}]\label{explorationpolairenemethi}
	Let $M$ be a real $3$-manifold. There exists finitely many triples $(\Gamma,L,P)$, where $\Gamma$ is a weighted graph and $L$ and $P$ are vectors in $(\mathbb{Z}_{\geq 0})^{V(\Gamma)}$, such that there exists a normal complex surface germ $(X,0)$  satisfying the following conditions:
	\begin{enumerate}
		\item The link of $(X,0)$ is homeomorphic to M.
		\item $(\Gamma,L,P)=(\Gamma_{\pi},L_{\pi},P_{\pi})$, where $\pi:(X_{\pi},E) \longrightarrow (X,0)$ is the minimal good resolution which factors through the blow-up of the maximal ideal and the Nash transform of $(X,0)$.
	\end{enumerate}
		
	\end{thm}
	
	A natural question to ask is if the statement of Theorem {\cite[Theorem A]{BFNP}} still holds  if we do not restrict ourselves to normal surface singularities?  Here, we give a negative answer to this question with the following result stated as Theorem \ref{thmM} in the introduction.
	  
	  We will need the notion of integral closure of ideals to state our result.

\begin{defi} \label{integralclosuredefi222}
Let $I$ be an ideal of a commutative ring $R$. The \textbf{integral closure of $I$} is the ideal, denoted $\bar{I}$, of all elements  $r$ in $R$ such that there exists $a_i \in I^{i}$ satisfying the equality 
		$$ r^n+a_1r^{n-1}+ \dots + a_{n-1}r+a_{1}=0.$$
	\end{defi}

\begin{thm}[{Theorem \ref{thmM}}]\label{belottofantininemethipichonbis}\label{belotonamathifantinipichon}		Let $(X,0)$ be a normal complex surface germ. Let $\Lambda_X$ be the infinite set of  integrally closed $\mathfrak{m}$-primary ideals of $\mathcal{O}_{X,0}$ .	There exists a family of complex surface germs with an isolated singularity $\{(X_I,0)\}_{I \in \Lambda_X}$ such that

	\begin{enumerate}
		\item $(X,0)$ is homeomorphic to $(X_I,0)$ for all $I \in \Lambda_X$ and is the common normalization of the germs $\{(X_I, 0)\}_{I \in \Lambda_X}$.	
		
		\item For all $I,J$ in $\Lambda_X$ we have that $ (\Gamma_{\pi_I},L_{\pi_I},P_{\pi_I}) = (\Gamma_{\pi_J},L_{\pi_J},P_{\pi_J})$ if and only if  $I=J$, where $\pi_I$ and $\pi_J$ are respectively the minimal good resolutions which factors through the blow-ups of the maximal ideal and the Nash transform of $(X_I,0)$ and $(X_J,0)$. \end{enumerate}

\end{thm}

In order to prove Theorem \ref{belottofantininemethipichonbis} we will need a preliminary result which gives a relationship between the inner rates and the integral closure of ideals (see \cite{lejeunetessier} for more knowledge on the integral closure of ideals). 

\begin{prop}\label{integralclosure2}
	Let $(X,0)$ be a complex surface germ and $I$, $J$ be two $\mathfrak{m}$-primary ideals of $\mathcal{O}_{X,0}$. If the inner rates functions $\mathcal{I}_I$ and $\mathcal{I}_J$ are equal then $\overline{I}=\overline{J}$.
\end{prop}
To prove Proposition \ref{integralclosure2} we need the following Theorem

\begin{thm}[See {\cite[Theorem 1.1]{leinear}} or {\cite[Theorem 2.1]{lejeunetessier}}] \label{jalabert}\label{sheafI} Let $(X,0)$ be a germ of complex space. Let $I$ be an ideal of $\mathcal{O}_{X,0}$ and $h$ be an element of $\mathcal{O}_{X,0}$. The following conditions are equivalent:
\begin{enumerate}
	\item $h \in \overline{I}$ 
	\item The function $h \circ \pi$ is a section of $I \cdot \mathcal{O}_{\overline{X_{I}}}$, where $\overline{\mathrm{BL}_I}:(\overline{X_I}, \overline{E_I}) \longrightarrow (X,0)$ is the normalized blow-up of $I$.

\end{enumerate}

 \end{thm}
 
\begin{proof}[Proof of Proposition \ref{integralclosure2}]
 	By Proposition \ref{integralclosure} there exists, up to isomorphism, a modification $$\pi_{IJ}:(\overline{X_{IJ}},\overline{E_{IJ}})\longrightarrow (X,0)$$ which is the common normalized blow-up of the ideals $I$ and $J$. Let us prove that $\overline{I}=\overline{J}$, let $h$ be an element of $\overline{I}$. By Theorem \ref{jalabert} the function $h \circ \pi_{IJ}$ is a section of $I.\mathcal{O}_{\overline{X_{IJ}}}$. By Proposition \ref{integralclosure} we have $I.\mathcal{O}_{\overline{X_{IJ}}}=J.\mathcal{O}_{\overline{X_{IJ}}}$ and it implies that $h \circ \pi_{IJ}$ is a section of $J.\mathcal{O}_{\overline{X_{IJ}}}$. Again, by Theorem \ref{jalabert} we get that $h$ is an element of $\overline{J}$. We proved that $\overline{I} \subset \overline{J}$. Symmetrically, we prove that $\overline{J} \subset \overline{I}$. \end{proof}
		\begin{rema}
		The converse of  Proposition \ref{integralclosure2} is false. Indeed, consider again the ideal $I$ of $\mathcal{O}_{\mathbb{C}^2,0}$ generated by $(x^2,y^2)$. The  integral closure $\overline{I}$ of the ideal $I$ is generated  by $(x^2,y^2,xy)$(see \cite[Proposition 3.4]{teissierencoreetencore}). Let $\omega_1=\mathrm{d}(x^2) \wedge \mathrm{d}(y^2)$ and $\omega_2=\mathrm{d}(x^2) \wedge \mathrm{d}(xy)$. Consider the modification  $$ \xymatrix{
   b:(X_2, E_{v_1} \cup E_{v_2})    \ar[r]^{ \ b_2}  & (X_1, E_{v_1}) \ar[r]^{b_1}   & (\mathbb{C}^2,0)    },$$ where $b_1$ is the blow-up of the origin of $\mathbb{C}^2$ and $b_2$ is the blow-up centered at the point $x^* \cap E_{v_1}$.  A direct computation using the holomorphic forms $\omega_1$ and $\omega_2$ shows that $\nu_{v_2}(I)= 5,m_{v_2}(I)=2$ and $\nu_{v_2}(\overline{I})=4,m_{v_2}(\overline{I})=2$. It follows that $$q_{v_2}^{I}= 2\neq \frac{3}{2} =q_{v_2}^{\overline{I}}.$$
	\end{rema}

In order to prove this Theorem \ref{belottofantininemethipichonbis}, we will also need the following result by Spivakovsky, which characterizes the Nash transform.		

\begin{thm}[{\cite[Theorem 1.2]{Spivakovsky1990}}]\label{spivaknash}
		A modification of $(X,0)$ factors through the Nash transform of $(X,0)$ if and only if it has no basepoints for the family of the polar curves associated to the linear projections on $\mathbb{C}^2$.
	\end{thm}

	Before going through the proof, let us give a concrete example of the construction we will use to prove Theorem \ref{belottofantininemethipichonbis}	.
	\begin{exam}
		Let $(X,0)=(\mathbb{C}^2,0)$ and consider the family of integrally closed $\mathfrak{m}$-primary ideal $\{I_n\}_{n \in \mathbb{N}^*}$ of $\mathcal{O}_{\mathbb{C}^2,0}$  generated by $G_n=(x^k y^{n-k},0 \leq k \leq n )$. The set of generators $G_n$ is not complete. Indeed, the points $(1,0)$ and $(\sqrt[n]{1},0)$, where $\sqrt[n]{1}$ is a root of unity not equal to $1$,  are sent to the point $(0,0,\dots,1)$ by the morphism $(x,y) \mapsto (y^n,x y^{n-1},x^2y^{n-2},\dots,x^{n-1}y,x^n)$.

		  Let us consider the complete set of generators $$\widetilde{G}_n=(y^{n+1},y^n,x y^{n-1},x^2y^{n-2},\dots,x^{n-1}y,x^n,x^{n+1})$$ for the ideal $I_n$. By Theorem \ref{immersion} the image of the germ $(\mathbb{C}^2,0)$ by the holomorphic map $$F_n(x,y)=(y^{n+1},y^n,x y^{n-1},x^2y^{n-2},\dots,x^{n-1}y,x^n,x^{n+1})=(z_0,z_1,\dots,z_{n+2}) \in \mathbb{C}^{n+3}$$ is a complex surface germ $(X_n,0)$ with an isolated singularity at the origin of $\mathbb{C}^{n+3}$ given by the equations:



$$
   z_{n+1}^k    {z_1}^{n-k}    z_{k+1}   = z_{n+2}^kz_0^{n-k}, \ 0 \leq k \leq n.
$$Again, by Theorem \ref{immersion} the inner rate function $\mathcal{I}_{\mathfrak{m}_n} : \mathrm{NL}(X_n,0) \longrightarrow \mathbb{R}_{+}^* \cup \infty $, where $\mathfrak{m}_n$ is the maximal ideal of $\mathcal{O}_{X_n,0}$,  is completely determined by the inner rate function $\mathcal{I}_{I_n} : \mathrm{NL}(\mathbb{C}^2,0) \longrightarrow \mathbb{R}_{+}^* \cup \infty $ and since the ideals $I_n$ are distinct and integrally closed then, by Proposition \ref{integralclosure}, the inner rate functions $\{\mathcal{I}_{I_n}\}_{n \in \mathbb{N}^*}$ are also distinct. We can conclude using the inner rates formula \ref{laplacien} that $(\Gamma_{\pi_n},L_{\pi_n},P_{\pi_n}) \neq (\Gamma_{\pi_m},L_{\pi_m},P_{\pi_m}),$ for all $n \neq m$ in $\mathbb{N}$, where $\pi_n$ and $\pi_m$ are respectively the minimal good resolutions which factor through the blow-up of the maximal ideal and the Nash transform of $(X_n,0)$ and $(X_m,0)$ respectively.
		
	\end{exam}

	 

\begin{proof}[Proof of Theorem \ref{belottofantininemethipichonbis}]		 Let $I$ be an ideal of $\Lambda_X$. By Theorem \ref{immersion}, there exists a complete system of generators $F=(f_1,\dots,f_n)$ for $I$ and a  complex surface germ  $(X_I,0)\subset (\mathbb{C}^n,0)$ with an isolated singularity at the origin of  $\mathbb{C}^n$ such that the map $F_I=(f_1,\dots,f_n):(X,0) \longrightarrow (X_I,0)$ is a homeomorphism and a modification of the complex surface germ $(X_I,0)$ embedded in $\mathbb{C}^n$. By Remark \ref{remarkokanormalization},  the map  $F_I$ is in particular a morphism of normalization. Furthermore, the map $F_I$ is such that the following diagram commutes
$$ \xymatrix{
    \mathrm{NL}(X,0) \ar[r]^{\tilde{F}_I}  \ar[rd]^{\mathcal{I}_{I}} & \mathrm{NL}(X_I,0)  \ar[d]^{\mathcal{I}_{\mathfrak{m}_{I}}} \\
      & \mathbb{R}_{+}^* \cup \infty
  },$$  where $\mathfrak{m}_{I}$ is the maximal ideal of $\mathcal{O}_{X_I,0}$.

   Let $\pi_I:(X_{\pi},E)\longrightarrow (X,0)$ be the minimal good resolution which factors through the blow-up of $I$ and the principalization of $\Omega^2_I$.  The good resolution $F_I \circ \pi_I: (X_{\pi},E)\longrightarrow (X_I,0)$ factors through the blow-up of the maximal ideal of $\mathcal{O}_{X_I,0}$ and the Nash transform of $X_I$. Indeed, the family of linear forms $P=\{p_i(z_1,\dots,z_n)=z_i \}_{ 1 \leq i\leq n}$ generate the ideal $\mathfrak{m}_I$ and the differentials $\{\mathrm{d}p_i \wedge \mathrm{d}p_j\}_{1 \leq i,j \leq n}$ generates $\Omega_{\mathfrak{m}_I}^2$. Obviously, the family $\Pi_P$ coincide with  the polar curves associated to the linear projections on $\mathbb{C}^2$ and we have that the $P^*:=(p_1 \circ F_I =f_1, \dots ,p_n \circ F_I=f_n)=F$ which is a complete system of generators for the ideal $I$. By definition of  the blow-up of $I$ and the principalization of $\Omega_I^2$, the good resolution  $\pi_I: (X_{\pi},E)\longrightarrow (X,0)$ has no basepoints for the families $L_{P^*}=L_F$ and $\Pi_{P^*}=\Pi_F$ and it follows that the good resolution $F_I \circ \pi_{I}$ has no basepoints for the families $L_P$ and $\Pi_P$. Then, by Proposition \ref{hironakablowupideal} and Theorem \ref{spivaknash} we get that $F_I \circ \pi_{I}$  factors through the blow-up of the maximal ideal of $\mathcal{O}_{X_I,0}$ and the Nash transform of $(X_I,0)$. It remains to prove that $F_I \circ \pi_{I}$ is the minimal good resolution which have these properties. Assume it is not the case. Then, there exists a good resolution $\sigma_I:(Y_I,Z_I) \longrightarrow (X_I,0)$ which factors through $\mathrm{BL}_{\mathfrak{m}_I}$ and the Nash transform of $X_I$ such that  $F_I \circ \pi_I$ factors through $\sigma_I$. On the other hand, since $\sigma_I$ is a good resolution, it factors through the normalization  $F_I:(X,0) \longrightarrow (X_I,0)$. It implies that there exists a good resolution $r_I:(Y_I,Z_I) \longrightarrow (X,0)$  such that $F_I \circ r_I=\sigma_I$.

   $$ \xymatrix{
    (Y_I,Z_I) \ar[r]^{\sigma_I}  \ar[rd]^{r_{I}} & (X_I,0)   \\
 (X_{\pi},E) \ar[r]^{\pi_I} \ar[u]^{}  & (X,0) \ar[u]^{F_{I}}
  },$$

    In order to prove that $\sigma_I$ factors through $\pi_I$ we  prove that $r_I$ factors through the blow-up of $I$ and the principalization of $\Omega^2_I$. Indeed, $r_I^*(F_{I}^*(\mathfrak{m}_I))=r_I^*(I)=\sigma_I^*(\mathfrak{m}_I)$ which implies that $\sigma_I^*(\mathfrak{m}_I) \otimes \mathcal{O}_{Y_I}=r_I^*(I) \otimes \mathcal{O}_{Y_I}.$ By Proposition
     \ref{hironakablowupideal}, the ideal sheaf $\sigma_I^*(\mathfrak{m}_I) \otimes \mathcal{O}_{Y_I}$ is locally principal because $\sigma_I$ factors through $\mathrm{BL}_{\mathfrak{m}_I}$. It follows that $r_I^*(I) \otimes \mathcal{O}_{Y_I}$ is also locally principal which implies by Proposition \ref{hironakablowupideal} that $r_I$ factors through the blow-up $\mathrm{BL}_I$. Furthermore, $\sigma_I$ has no basepoints for the family $\Pi_P$ which implies that $r_I$  has no basepoints for the family $\Pi_F$, but it is not sufficient. In order for $r_I$ to factor through the principalization of $\Omega^2_I$ it needs to have no basepoints for any family $\Pi_{G}$, where $G$ is a precomplete system of generators for $I$. Let $G=(g_1,\dots,g_l)$ be a precomplete system of generators for $I$.  Since $G$ is precomplete then for every couple $(\alpha,\beta)  \in \mathbb{P}(\mathbb{C}^k) \times \mathbb{P}(\mathbb{C}^k) $ we have \begin{equation}\label{4.6}
      r_I^*(\mathrm{d}F_{\alpha} \wedge \mathrm{d}F_{\beta})=\sum_{i \neq j} \psi_{ij}^{\alpha \beta} \pi^*(\mathrm{d}g_i \wedge \mathrm{d}g_j). \end{equation}Let $p$ be a smooth point of an irreducible component $Z_v$ of $Z_I$. Let $(x,y)$ be a local coordinate system of $Y_I$ centered at $p$ such that $x=0$ is the equation of $Z_v$. By Lemma \ref{nu}  we have 
     $$r_I^*(\mathrm{d}F_{\alpha} \wedge \mathrm{d}F_{\beta})=x^{\nu_v(I)} S_{\alpha \beta}(x,y) \mathrm{d}x \wedge \mathrm{d}y, \ \text{with} \ S_{\alpha \beta}(0,y) \neq 0, $$ since $r_I$ has no basepoints for the family $\Pi_F$, there exists a Zariski open set $O_p$ of $\mathbb{P}(\mathbb{C}^k) \times \mathbb{P}(\mathbb{C}^k)$ such that $S_{\alpha \beta}(0,0) \neq 0$ for all $(\alpha,\beta) \in O_p$. It follows from the equality (\ref{4.6}) that there exists $i_0,j_0 \in \{1,\dots,l\}$ such that 
      $$ r_I^*(\mathrm{d}g_{i_0} \wedge \mathrm{d}g_{j_0})=x^{\nu_v(I)}S_{0}(x,y) \mathrm{d}x \wedge \mathrm{d}y $$ with $S_0(0,0)\neq 0$.  It implies that there exists a Zariski open set $W$ such that for any  $(\alpha,\beta)$  in $W$ $$r_I^{*}(\mathrm{d}G_{\alpha} \wedge \mathrm{d}G_{\beta})=x^{\nu_v(I)}U_{\alpha \beta}(x,y) \mathrm{d}x \wedge \mathrm{d}y,$$ where $U_{\alpha \beta}$ is a unit of $\mathbb{C}\{x,y\}$.  It proves that $p$ is not a basepoint for the family $\Pi_G$. We can do the same if $p$ is a double point.

Now, we have proved that $r_I$ factors through  the principalization of $\Omega^2_I$. By definition, it implies that $r_I$ factors through $\pi_I$. Finally, we proved that $\pi_I \circ F_I$ is indeed the minimal good resolution which factors through the blow-up of $\mathfrak{m}_I$ and the Nash transform of $(X_I,0)$.

 By Proposition \ref{integralclosure2}, for any two distinct integrally closed ideals $I$ and $J$ of $\Lambda_X$ we have  $\mathcal{I}_{\mathfrak{m}_I} \neq \mathcal{I}_{\mathfrak{m}_J}$ and it follows from Proposition \ref{nashvsrates} and the inner rates formula \ref{laplacien} that $(\Gamma_{\pi_I},L_{\pi_I},P_{\pi_I}) \neq (\Gamma_{\pi_J},L_{\pi_J},P_{\pi_J})$.
 

 \end{proof}

		\subsection{Application to Outer Lipschitz geometry} As a corollary of Theorem \ref{belottofantininemethipichonbis}, we can construct infinitely many complex surface germs with an isolated singularity that are not metrically equivalent in the Bilipschitz sense but have a common normalization.
		
		First, let us recall some definitions and results about the bilipschitz geometry of complex germs (See \cite{pichonintroduction} for more details). 

\begin{defi}
Let $(X,0)$ be a complex surface germ embedded in $\mathbb{C}^n$. The \textbf{outer metric} $\mathrm{d}_o$ on $X$ is the metric induced by the ambient Euclidian metric, i.e., for all $x,y \in X$, $\mathrm{d}_o(x,y)=|| x-y ||_{\mathbb{C}^n}$.
The \textbf{inner metric} $\mathrm{d}_i$ on $X$ is the arc length metric induced by the ambient riemannian metric.
\end{defi}
 \begin{defi}
 	Let $(M,d)$ and $(M',d')$ be two metric spaces. A map $f:M \longrightarrow M'$ is a \textbf{bilipschitz homeomorphism} if $f$ is a homeomorphism and there exists a real constant $K \geq 1$ such that 
 	
 	 $$ \frac{1}{K}d(x,y) \leq d'(f(x),f(y)) \leq Kd(x,y), \ \forall x,y \in M.$$
 \end{defi}
		\begin{defi}\label{bilequiv}
			Two analytic germs $(X,0)\subset (\mathbb{C}^n,0)$ and $(X',0)\subset (\mathbb{C}^m,0)$ are \textbf{inner Lipschitz equivalent} (resp. \textbf{outer Lipschitz equivalent}) if there exists a germ of bilipschitz homeomorphism $\psi:(X,0) \longrightarrow (X',0)$	with respect to the inner (resp. outer) metric. The equivalence class of the germ $(X,0) \subset (\mathbb{C}^n,0)$ for this equivalence relation is called the \textbf{inner Lipschiz geometry} (resp. \textbf{outer Lipschitz geometry}) of $(X,0)$.	\end{defi}
			
			One of the main motivations for studying complex surface germs up to bilipschitz equivalence is that the biholomorphic type of a complex germ determines its bilipschitz geometry. Another motivation follows from a theorem by Mostowski \cite{Mostowski1985a}, which states that the set of inner (resp. outer) Lipschitz geometries  is countable. While a complete classification of complex surface germs with isolated singularities has been given in \cite{BNP} the classification of complex surface germs up to outer Lipschitz equivalence is still open.

			 The next result aims to understand the set of outer Lipschitz geometries of complex surface germs with isolated singularities.
			
			\begin{coro}[Corollary \ref{coroM}]\label{belottofantinipichonlipschitzbis}
				Let $(X,0)$ be a normal complex surface germ. Let $\Lambda_X$ be the infinite set of  integrally closed $\mathfrak{m}$-primary ideals of $\mathcal{O}_{X,0}$ .	There exists a family of complex surface germs with an isolated singularity $\{(X_I,0)\}_{I \in \Lambda_X}$ such that

	\begin{enumerate}
		\item $(X,0)$ is homeomorphic to $(X_I,0)$ for all $I \in \Lambda_X$ and is the common normalization of the germs $\{(X_I, 0)\}_{I \in \Lambda_X}$.	
		
		\item For all $I,J \in \Lambda_X$ the surfaces $(X_I,0)$ and $(X_J,0)$ are outer Lipschitz equivalent if and only if $I=J$.
\end{enumerate}

							\end{coro} 
	Corollary \ref{belottofantinipichonlipschitzbis} is a direct consequence of Theorem \ref{belottofantininemethipichonbis} and the following result of Neumann and Pichon

		\begin{thm}[{\cite[Theorem 1.2]{NP2016}}]\label{NP16}
		Let $(X,0)$ be a complex surface germ with an isolated singularity. Then the outer Lipschitz geometry of $X$ determines the triple $(\Gamma_{\pi},L_{\pi},P_{\pi})$ where $\pi$ is the minimal good resolution which factors through the blow-up of the maximal ideal of $\mathcal{O}_{X,0}$ and the Nash transform of $(X,0)$.
	\end{thm}

	\thispagestyle{empty}

\bibliographystyle{alpha}

\addcontentsline{toc}{chapter}{Bibliography}
\bibliography{biblio}

\end{document}